\documentclass[10pt]{amsart}
\usepackage{palatino, amssymb, amsfonts, latexsym, mathrsfs}
\usepackage{amssymb, amsmath, amscd}
\input xy
\xyoption{all}

\newtheorem{Theorem}{Theorem}[subsection]
\newtheorem*{Theorem*}{Theorem}
\newtheorem{lemma}{Lemma}[subsection]
\newtheorem{prop}{Proposition}[subsection]

\newtheorem{cor}{Corollary}[subsection]

\theoremstyle{remark}

\newtheorem{remark}{Remark}[subsection]

\newcommand{\PP}{\mathbb{P}}
\newcommand{\E}{\widehat{\mathcal E}}

\newtheorem{theorem}{Theorem}[subsection]
\theoremstyle{definition}
\newtheorem{definition}[theorem]{Definition}
\newtheorem{example}[theorem]{Example}

\newcommand{\gh}{\mathfrak h}
\newcommand{\hreg}{\mathfrak h^{\text{reg}}}
\newcommand{\Hc}{\mathcal H_c}
\newcommand{\KZ}{\text{KZ}}
\newcommand{\OHc}{\mathcal O_{\mathcal H_c}}
\newcommand{\Lieg}{\mathfrak g}
\newcommand{\Hil}{\text{Hilb}^n(\mathbb C^2)}
\newcommand{\tL}{\tilde{\Lambda}}
\newcommand{\GL}{\text{GL}}
\newcommand{\tht}{\theta}
\newcommand{\Ac}{\tilde{\mathscr A}_c} 
\newcommand{\Aci}{\tilde{\mathscr A}_{c,i}}
\newcommand{\To}{\mathring{T}}
\newcommand{\Msn}{\mathcal M^{\text{sn}}}
\newcommand{\Msnt}{\tilde{\mathcal M}^\text{sn}}

\newcommand{\tAc}{\tilde{\mathscr A}_c}

\numberwithin{equation}{section}

\begin{document}

\title{Microlocal KZ functors and rational Cherednik algebras.}

\author{Kevin McGerty}
\address{Mathematical Institute, University of Oxford. }

\date{October 2011.}

\begin{abstract}
Following the work of Kashiwara-Rouquier and Gan-Ginzburg, we define a family of exact functors from category $\mathcal O$ for the rational Cherednik algebra in type $A$ to representations of certain ``coloured braid groups'' and calculate the dimensions of the representations thus obtained from standard modules. To show that our constructions also make sense in a more general context, we also briefly study the case of the rational Cherednik algebra corresponding to complex reflection group $\mathbb Z/l\mathbb Z$. 
\end{abstract}

\maketitle

\section{Introduction}

\subsection{}
This paper is inspired by two beautiful constructions in the theory of rational Cherednik algebras of type $A$. The first is the KZ functor studied in \cite{GGOR} which relates representations in the category $\mathcal O_c$ to representations of Hecke algebras, and the second is the localization theorem of Kashiwara-Rouquier (see also the work of Gan and Ginzburg \cite{GG}) which allows us to realize the category $\mathcal O_c$ as a category of modules for a ``microlocal'' sheaf of rings $\mathcal A_c$ on the Hilbert scheme $\Hil$ of $n$ points in the plane (with certain restrictions on the value of the deformation parameter $c$).

In the context of \cite{KR}, the modules in the category $\mathcal O_c$ correspond to sheaves of $\mathcal A_c$-modules supported on a certain reducible Lagrangian subvariety $Z$ of $\Hil$ first considered by Grojnowski \cite{Gr}. Using this microlocal point of view, we construct for each component of $Z$ an exact functor from $\mathcal O_c$ to a category of local systems, and we show that the $KZ$-functor naturally corresponds to one of them.

On the level of dimensions, these functors recover the characteristic cycle as defined by Gan and Ginzburg \cite{GG}, and for standard modules we are able to calculate these dimensions. This calculation is already known by combining work of Ginzburg, Gordon and Stafford \cite{GGS} with earlier work of Gordon and Stafford \cite{GS}, but our calculations are self-contained, and in particular make no use of Haiman's work on the Hilbert scheme. 

In section \ref{RCAstuff} we recall the definition of the rational Cherednik algebra in type $A$ and the construction of the $KZ$-functor. In section \ref{deformed} we review the quantum Hamiltonian reduction construction of Gan-Ginzburg. In section \ref{Micropi1} we study the Lagrangian subvariety $Z$ of the Hilbert scheme and exhibit its smooth locus. In section \ref{MKZfunctors} we define for each component $Z_\lambda$ of $Z$ an exact functor $KZ_\lambda$ on category $\mathcal O_c$ generalizing the $KZ$-functor. In section \ref{CCcomputation} we compute the characteristic cycles of standard modules, which give the dimensions of our functors $KZ_\lambda$ on these modules. In section \ref{HeckeAlgebras} we show using $\mathcal D$-module techniques that the original $KZ$-functor descends to a functor to representations of a Hecke algebras (as opposed to just a braid group). Finally, to give an example of our techniques for a complex reflection group, in section \ref{cyclic} we consider the case of the complex reflection group $\mathbb Z/l\mathbb Z$. Here we find a direct connection to the theory of cyclotomic $q$-Schur algebras, while hints of a similar connection to the classical $q$-Schur algebra are suggested by our results in the case of the symmetric group. We hope to return to this issue in a later paper. Finally in the Appendix we give a review of material on twisted $\mathcal D$-modules and equivariance, and a brief discussion of the index theorem which is needed in our calculation of characteristic cycles of standard modules.

\noindent
\textit{Acknowledgements}: The author would like to thank Kobi Kremnitzer and Tom Nevins for useful discussions on $\mathcal D$-modules, and Ian Grojnowski, Iain Gordon, and Toby Stafford for introducing him to the world of rational Cherednik algebras along with much encouragement while this paper was being written. The author also gratefully acknowledges the support of a Royal Society University Research Fellowship.

\section{Rational Cherednik algebra and the $\KZ$ functor.}
\label{RCAstuff}
\subsection{}
Let $\gh$ be the permutation representation of the symmetric group $W=S_n$ on $n$ letters. Thus $\gh = \mathbb C^n$ and $W$ acts by permuting the coordinates in the obvious way. The rational Cherednik algebras\footnote{We have set the parameter $t$ of \cite{EG} to be $1$.} $\Hc$ of type $A$, for $c \in \mathbb C$, are a family of algebras giving a flat deformation of $\mathcal D(\gh)\rtimes \mathbb C[W]$, the smash product of the group algebra with the algebra of differential operators on $\mathfrak h$.

Let $\{y_1,y_2,\ldots,y_n\}$ be the standard basis of $\gh$, and let $\{x_1,x_2,\ldots, x_n\}$ be the corresponding dual basis of $\gh^*$. If we let $s_{ij} \in W$ denote the transposition which interchanges $i$ and $j$, then $\Hc$ is generated by $\{x_i,y_i,s_{ij}: 1\leq i,j \leq n, i \neq  j\}$ subject to relations:

\[
\begin{split}
s_{ij}x_i &= x_js_{ij}, \quad s_{ij}y_i = y_js_{ij}, \quad 1 \leq i,j, \leq n, i \neq j, \\
[y_i,x_j] &= c s_{ij}, \quad [x_i,x_j] = [y_i,y_j] =0, \quad \forall 1 \leq i,j \leq n, i \neq j\\
[y_k, x_k] &= 1 -c\sum_{i \neq k} s_{ik}.
\end{split}
\]
Clearly $\mathcal H_0$ is just $\mathcal D(\gh) \rtimes \mathbb C[W]$. In general, one may filter $\Hc$ by placing $\gh$ in degree $1$, $\gh^*$ and $\mathbb C[W]$ in degree $0$. The associated graded algebra is then isomorphic to $\mathbb C[\gh\times \gh^*]\rtimes \mathbb C[W]$, and the canonical map from $\Hc$ is known to be a vector space isomorphism \cite[Theorem 1.3]{EG}. This yields a triangular decomposition for $\Hc$: multiplication in $\Hc$ induces a vector space isomorphism
\[
\mathbb C[\gh^*] \otimes \mathbb C[W] \otimes \mathbb C[\gh] \to \Hc.
\]

Motivated by the analogy between this decomposition and the well-known triangular decomposition of a semisimple Lie algebra, one can consider the category $\OHc$ of representations for $\mathcal H_c$ which are locally finite for the subalgebra $\mathbb C[\gh^*]$. This category was introduced in \cite{DO} and has been extensively studied (see for example \cite{BEG}, \cite{GGOR}). A module $M$ in category $\OHc$ is said to have \textit{type} $\bar{\lambda}$ for $\bar{\lambda} \in \gh/W = \text{Spec}(\mathbb C[\gh^*]^W)$ if for any element $P \in \mathbb C[\gh]^W$ the operator $P - P(\bar{\lambda})$ is locally nilpotent. The full subcategory of $\OHc$ consisting of objects of type $\bar{\lambda}$ is denoted $\OHc(\bar{\lambda})$ and $\OHc$ splits as a direct sum of the subcategories $\OHc(\bar{\lambda})$. We shall focus on the category $\OHc(0)$, where $\mathbb C[\gh]^W$ acts locally nilpotently, and thus for convenience we shall denote it by $\mathcal O_c$. It is abelian and every object has finite length. Note that the results of Bezrukavnikov and Etingof on induction/restriction functors allow one to reduce the study of $\OHc(\bar{\lambda})$ to $\mathcal O_{\mathcal H'_c}(0)$ for $\mathcal H'_c$ a rational Cherednik algebra corresponding to a parabolic subgroup of $W$, so that questions about $\OHc$ can be reduced to ones for $\mathcal O_{\mathcal H_c}(0)$ (see \cite[Corollary 3.3]{BE} for a precise statement). 

\subsection{}
In order to relate $\Hc$ to the Hibert scheme, we will also need its spherical subalgebra, $e\Hc e$, where $e = |W|^{-1}\sum_{w \in W} w$ is the idempotent in $\mathbb C[W]$ corresponding to the trivial representation. The map $a \mapsto ae$ yields embeddings from $\mathbb C[\gh]^W \to \mathbb C[\gh]^We \hookrightarrow e\Hc e$ and from $\mathbb C[\gh^*]^W \to \mathbb C[\gh^*]^We \hookrightarrow e\Hc e$. Moreover, if $e\mathcal H_ce$ is simple, then their images generate $e\Hc e$ \cite[Theorem 4.6]{BEG}. It is easy to see that $e \Hc e$ is simple whenever $\Hc$ is simple\footnote{However it is also shown in \cite[Theorem 4.10]{BE} that $e\Hc e$ is simple for all $c \in (-1,0)$, a range in which $\Hc$ is not always simple.}, and by \cite[Theorem 3.1]{BEG} the algebra $\Hc$ is simple algebra exactly when the associated finite Hecke algebra $H_W(q)$ is semisimple, where $q = e^{2\pi i c}$. One has the obvious functor from $\mathcal H_c$-modules to $e\mathcal H_c e$-modules given by $M \mapsto eM$. This is easily seen \cite[Lemma 4.1]{BEG} to be a Morita equivalence when $\mathcal H_c$ is simple.

\subsection{}
We now recall the construction of the $\KZ$-functor of \cite{GGOR}. Let
\[
\hreg = \{v \in \gh: |W\cdot v| = |W|\},
\]
be the subset of $\gh$ consisting of those points whose coordinates are pairwise distinct. If we set $\delta = \prod_{i<j}(x_i-x_j) \in \mathbb C[\gh]$ then $\hreg = \{ v \in \gh: \delta(v) \neq 0\}$, hence it is a Zariski-open subset of $\gh$. The key to the construction is the Dunkl homomorphism from $\Hc$ to $\mathcal D(\hreg)\rtimes \mathbb C[W]$. This is given by the obvious embedding of the subalgebra $\mathbb C[\mathfrak h]\rtimes \mathbb C[W]$, and by sending $y \in \gh$ to the operator
\[
T_y = \partial_y + c\sum_{i<j} \frac{y_i-y_j}{x_i-x_j}(s_{ij}-1) \in \mathcal D(\hreg)\rtimes \mathbb C[W].
\]
It is straight-forward to check that $S = \{\delta^k : k \in \mathbb Z_{>0}\}$ is an Ore set in $\mathcal H_c$, so that it makes sense to localize $\mathcal H_c$ at $S$. We refer to this as localizing $\mathcal H_c$ to $\mathfrak h^{\text{reg}}$ and denote the resulting algebra by $\mathcal H_{c|\mathfrak h^{\text{reg}}}$. Similarly we may localize an $\mathcal H_c$-module at $S$ to obtain an  $\mathcal H_{c|\mathfrak h^{\text{reg}}}$-module denoted $M_{|\mathfrak h^{\text{reg}}}$. We then have the following result:

\begin{prop}\cite[\S 4]{EG}
The assignment $w \mapsto w$, $x \mapsto x$ and $y \mapsto T_y$ extends to an algebra homomorphism $\Theta_c$ which embeds $\Hc$ into $\mathcal D(\hreg)\rtimes \mathbb C[W]$. Moreover, this map becomes an isomorphism after localizing to $\hreg$.
\end{prop}

Note that the final part of the proposition is immediately clear from the first part and the explicit formula for $T_y$.  Now given any module $M$ in the category $\mathcal O_{\mathcal H_c}$, via the above isomorphism, we may view $M_{|\mathfrak h^{\text{reg}}}$ as a module for $\mathcal D(\hreg)\rtimes \mathbb C[W]$-module, and hence, since $\hreg$ is an affine variety, as a $W$-equivariant module for the sheaf of differential operators $\mathcal D_{\hreg}$ on $\hreg$. Now it is not hard to show that $M$ is finitely generated as a $\mathbb C[\gh]$-module (see for example Lemma $2.5(i)$ of \cite{BEG}), thus $M_{|\hreg}$ is coherent as an $\mathcal O_{\hreg}$-module. It follows that the associated $\mathfrak D_{\hreg}$-module must be a vector bundle with a flat connection, and hence it gives rise to a local system on $\hreg$. Since this local system is $W$-equivariant, it descends to give a local system on $\hreg/W$. But now if $x_0 \in \hreg$ is any point, it is known that $\pi_1(\hreg/W, x_0)$ is isomorphic to the braid group $\mathcal B_W$, and hence we get a functor
\[
\KZ \colon \OHc \to \text{Rep}(\mathcal B_W),
\]
by for example applying the deRham functor to $M_{\hreg}$ (alternatively one could take the dual representation given by the solution complex). One of the main results of \cite{GGOR} shows that the image of this functor lies in the much smaller category of representations of the associated Hecke algebra.

\begin{Theorem}(\cite[Theorem 5.13]{GGOR})
The functor $\KZ$ factors through the category of representations of the Hecke algebra $\mathcal H_q$ where the parameter $q$ is specialized to $e^{2\pi i c}$.
\end{Theorem}
\noindent
We will give a new proof of this theorem in Section \ref{MKZfunctors} in the $\mathcal D$-module context of \cite{GG}.

\section{The deformed Harish-Chandra homomorphism and the Hilbert scheme.}
\label{deformed}

\subsection{}
\label{almostcommuting}
In this section we recall the localization of the rational Cherednik algebra in type $A$, as studied in \cite{KR} and the related work \cite{GG}. Let $V$ be an $n$-dimensional vector space and let $\Lieg = \mathfrak{gl}(V)$ be the Lie algebra of linear maps from $V$ to itself. The almost commuting variety is the variety of quadruples
\[
\mathcal M = \{(X,Y,i,j) \in \Lieg \times \Lieg \times V \times V^*: [X,Y]+ji = 0\},
\]
where $ji$ denotes the rank one linear map $v \mapsto j(v).i$. Using the trace form we may naturally identify $\Lieg$ with $\Lieg^*$, so that $\mathcal M$ may be viewed as a subvariety of $T^*\mathfrak G$ where we set $\mathfrak G = \Lieg\times V$. It is known \cite[Theorem 1.1]{GG} that $\mathcal M$ is an equidimensional variety with $n+1$ components of dimension $n^2 +2n$. We also consider the set
\[
\mathcal M_{\text{nil}} = \{(X,Y,i,j) \in \mathcal M: Y \text{ is nilpotent}\}.
\] 
This is a (reducible) Lagrangian subvariety. Finally we consider the Zariski open set (given by a stability condition)
\[
\text{U} = \{(X,Y,i,j) \in T^*\mathfrak G: \mathbb C\langle X,Y\rangle i = V\}
\]
where $\mathbb C\langle X,Y\rangle$ is the subalgebra of $\text{End}(V)$ generated by $X$ and $Y$. If we let $\mathcal M^s = \mathcal M \cap \text{U}$, then it can be shown that $\text{GL}(V)$ acts freely on $\mathcal M^s$, and the quotient is isomorphic to $\text{Hilb}^n(\mathbb C^2)$, the Hilbert scheme of $n$ points in the plane. At the level of sets the isomorphism can be seen as follows: if $(X,Y,i,j) \in \mathcal M^s$ then the stability condition in fact forces $j$ to be zero, so that $X$ and $Y$ commute. Thus they define a homomorphism from $\mathbb C[x,y] \to \text{End}(V)$, and the kernel of this homomorphism is a codimension $n$ ideal of $\mathbb C[x,y]$. Note that this realization of $\text{Hilb}^n(\mathbb C^2)$ equips it with a natural ample line bundle corresponding to the $G$-equivariant line bundle $\mathcal M^s \times \text{det}$ on $\mathcal M^s$. Let $\Msn$ denotes the intersection of $\mathcal M^s$ with $\mathcal M_{\text{nil}}$. It is $G$-stable its quotient is a Lagrangian subvariety of $\Hil$ which we denote by $Z$.

\subsection{}
\label{DmodsonG}
We now relate the rational Cherednik algebra to twisted $\mathcal D$-modules. For the convenience of the reader, and to fix notation, in Appendix \ref{twistingstuff} we review a construction of the kinds of twisted rings of differential operators we need, along with notions of equivariance. Let $V^\circ = V\backslash \{0\}$, let $\mathfrak G^\circ = \Lieg \times V^\circ$, and let $\mathfrak X = \Lieg\times \mathbb P$ where $\mathbb P = \mathbb P(V)$ is the projective space of lines in $V$. Clearly $\mathfrak G^\circ$ is a principal $\mathbb G_m$-bundle over $\mathfrak X$, so that given any $\lambda \in \text{Hom}(\text{Lie}(\mathbb G_m),\mathbb C)$, we may consider modules for the corresponding sheaf $\mathcal D_{\mathfrak X,\lambda}$ of $\lambda$-twisted differential operators on $\mathfrak X$, or equivalently $(\mathbb G_m,\lambda)$-twisted equivariant $\mathfrak D_{\mathfrak G^0}$-modules. Note that we may describe the cotangent bundle of $\mathfrak X$ as 
\[
\{(X,Y,i,j) \in T^*\mathfrak G^\circ: \langle i,j\rangle =0\}/\mathbb C^\times, 
\]
where $\mathbb C^\times$ acts via $t\cdot(X,Y,i,j) = (X,Y,ti,t^{-1}j)$ and as above we have identified $T^*\mathfrak G^\circ$ with $\Lieg\times\Lieg\times V^\circ\times V^*$. The function $(X,Y,i,j) \mapsto [X,Y]+ij$ is fixed by this $\mathbb C^\times$-action, and moreover if $(X,Y,i,j) \in \mathcal M$ then
\[
\langle i,j\rangle = \text{Tr}(ji) = - \text{Tr}([X,Y]) = 0,
\]
thus 
\[
\mathcal M \cap \{(X,Y,i,j) \in T^*\mathfrak G^\circ :\langle i,j \rangle = 0\} = \mathcal M \cap T^*\mathfrak G^\circ.
\]
Let $\mathfrak M$ denote the quotient of $\mathcal M \cap T^*\mathfrak G^\circ$ by $\mathbb C^\times$, a subvariety of $T^*\mathfrak X$, and similarly let $\Lambda$ be the quotient of $\mathcal M_\text{nil}\cap T^*\mathfrak G^\circ$ (a Lagrangian subvariety of $T^*\mathfrak X$). If $(X,Y,i,j) \in \mathcal M^s$ then $i \neq 0$ and $\langle i,j\rangle =0$ (since $j=0$), so that $\mathcal M \cap T^*\mathfrak G^\circ$ contains $\mathcal M^s$ and $\Msn$. We denote their quotients in $T^*\mathfrak X$ by $\mathfrak M^s$ and $\Lambda^s$.  
 
The action of $\text{SL}(V)$ on $\mathfrak G^\circ$ commutes with the $\mathbb G_m$-action, so that we may also consider the category of $\text{SL}(V)$-equivariant $\mathcal D_{\mathfrak X,\lambda}$-modules. Now if we write $\mathcal Z\cong \mathbb G_m$ for the centre of $\text{GL}(V)$, then the action of $\mathbb G_m$ on $\mathfrak G^\circ$ is clearly the restriction to $\mathcal Z$ of the action of $\text{GL}(V)$ on $\mathfrak G^\circ$. We may thus also consider the category of $(\text{GL}(V),c.\text{tr})$-equivariant $\mathcal D$-modules on $\mathfrak G^\circ$ where as usual $\text{tr}$ denotes the trace character on $\Lieg$. Note that given $c \in \mathbb C$, the corresponding $\lambda$ above for $\text{Lie}(\mathbb G_m)=\text{Lie}(\mathcal Z)$ is given by $\lambda = c.\text{tr}_{|\text{Lie}(\mathcal Z)}$, and following \cite{GGS} we will from now on write $\mathcal D_{\mathfrak X,c}$ for $\mathcal D_{\mathfrak X,\lambda}$. Now since $\text{tr}$ vanishes on $\mathfrak{sl}(V)$ and $\text{GL}(V) = \mathcal Z.\text{SL}(V)$, this category is equivalent to the category of $\text{SL}(V)$-equivariant $\mathcal D_{\mathfrak X,c}$-modules on $\mathfrak X$ and moreover this equivalence respects holonomic modules. Note that the category of $(\text{GL}(V),c.\text{tr})$-equivariant modules is considered in \cite{KR}, while the category of $\text{SL}(V)$-equivariant $\mathcal D_{\mathfrak X,c}$-modules is used in \cite{GG}. For more details concerning this equivalence, at least at the level of algebras of global sections, see \cite[\S 6]{GGS}. 

We also remark that if we write $\mathcal D_c(\mathfrak X)$ for the global sections of $\mathcal D_{\mathfrak X,c}$, then provided $n>3$ or $n=2$ and $2c\notin \mathbb Z_{\leq 0}$, it is shown in \cite[\S6.2]{GGS} that 
\[
\mathcal D_c(\mathfrak X) \cong \big(\mathcal D(\mathfrak G)/(\mathcal D(\mathfrak G)(\nu-c.\text{tr})(\mathfrak z))\big)^{\mathcal Z}
\]
where $\nu \colon \mathfrak{gl}(V) \to \Theta_{\Lieg \times V}$ is the quantized moment map for the $\text{GL}(V)$-action, and $\mathfrak z$ is the Lie algebra of $\mathcal Z$. Thus one may effectively work on $\mathfrak G$ or $\mathfrak G^\circ$.

\subsection{}
\label{deformedHC}
The action of $\text{SL}(V)$ on $\mathfrak X$ induces an infinitesimal action of $\mathfrak{sl}_n$, and hence a homomorphism from $\mathfrak{sl}_n$ to $\Theta_{\mathfrak X}$  the Lie algebra of vector fields on $\mathfrak X$, which extends naturally to a map $\tau \colon \mathcal U(\Lieg) \to \mathcal D_\mathfrak X$.  Similarly, we obtain a homomorphism $\tau_c\colon \mathcal U(\Lieg) \to \mathcal D_{\mathfrak X,c}$ from the enveloping algebra to the sheaf of $c$-twisted differential operators. Let $\mathfrak g_c$ denote the image of $\mathfrak{sl}(V)$ under $\tau_c$.  Since projective space is known to be $\mathcal D$-affine for $c \notin \{-k/n: k \in \mathbb Z_{>0}\}$, for these values the category of $\mathcal D_{\mathfrak X,c}$-modules is equivalent to the category of modules for its algebra of global sections $\mathcal D_c(\mathfrak X)$.

One of the main results of \cite{GG} is the construction of a ``deformed Harish-Chandra homomorphism'' 
\[
\Phi_c \colon (\mathcal D_c(\mathfrak X)/\mathcal D_c(\mathfrak X)\cdot \Lieg_c)^{\text{ad}(\Lieg_c)} \to e\mathcal H_c e.
\]
The map $\Phi_c$ is a filtered isomorphism sending $\mathbb C[\Lieg]^{\text{Ad} G}$ to $\mathbb C[\mathfrak h]^{S_n}$ and $\mathfrak Z$, the algebra of constant coefficient differential operators on $\Lieg$, to $\text{Sym}(\mathfrak h)^{S_n}$ (see Theorem $1.5$ of \cite{GG}).

Since we will later examine the connection between the original $KZ$-functor and our microlocal functors, we review briefly the construction of the map $\Phi_c$. It consists of two steps: first there is the Dunkl embedding which we have already discussed. This gives a map $\Theta_c \colon \mathcal H_c \to \mathcal D(\mathfrak h^{\text{reg}})\rtimes \mathbb C[W]$. On the other hand one can construct a ``radial parts'' homomorphism 
\[
\mathfrak R_c\colon (\mathcal D_c(\mathfrak X)/\mathcal D_c(\mathfrak X)\cdot \Lieg_c)^{\text{ad}(\Lieg_c)} \to \mathcal D(\mathfrak h^{\text{reg}})^W.
\]
To construct this map let $\mathfrak X^{\text{reg}}$ be the open subset of $\mathfrak X$ consisting of pairs $(X,\ell)$ where the nonzero vectors in $\ell$ are cyclic for $X$ and let $\mathfrak G^\text{reg}$ denote its preimage in $\mathfrak G^\circ$. If we pick a holomorphic volume form $\Omega \in \bigwedge^n(V^*)$ then the function 
\[
s\colon \mathfrak G \to \mathbb C, \quad (X,v) \mapsto \langle v\wedge Xv\wedge \ldots \wedge X^{n-1}v,\Omega \rangle,
\]
shows that $\mathfrak G^\text{reg}$ is a principal affine open of $\mathfrak G$, and hence $\mathfrak X^{\text{reg}}$ is the complement of a principal divisor in $\mathfrak X$. It is easy to check that the action of $\text{PGL}(V)$ on $\mathfrak X^{\text{reg}}$ is free (see \cite[\S 5.3]{BFG} or Proposition \ref{smoothlocus} below), making it into a principal $\text{PGL}(V)$-bundle over $\mathfrak h/W$. Thus by descent we see that 
\[
(\mathcal D_c(\mathfrak X^{\text{reg}})/\mathcal D_c(\mathfrak X^{\text{reg}})\cdot \Lieg_c))^{\text{PGL}(V)} \cong \mathcal D(\mathfrak h/W), \quad (c \in \mathbb C).
\]
where the isomorphism is given explicitly by 
\[
\mathfrak R_c(D) = s^c(D_{|\mathcal O(\mathfrak X^{\text{reg}},c)})s^{-c}, \quad D \in \mathcal D(\mathfrak X)^G.
\]
Now we have inclusions $\mathcal D(\mathfrak h)^W \subset \mathcal D(\mathfrak h/W) \subset \mathcal D(\mathfrak h^{\text{reg}})^W$, and it can be shown that the image in $\mathcal D(\mathfrak h^\text{reg})^W$ of the composite of  $\mathfrak R_c$ with this inclusion is precisely the image of the spherical subalgebra of $\mathcal H_{c-1}$ under the Dunkl homorphism. A careful discussion of this fact is given in the Appendix to \cite{GGS}. However the map $\Psi_c$ which we work with is, as in \cite[\S 5.4]{BFG}, a twist of this radial parts map by the discriminant $\delta$: we set
\[
\Psi_c(D) = m_\delta^{-1}\circ \mathfrak R_c(D)\circ m_\delta,
\]
where $m_f$ denotes the operator of multiplication by the function $f$. The image $\text{im}(\Psi_c)$ of this twisted radial parts map is exactly the image of $e\mathcal H_c e$ under the Dunkl homomorphism, and the composition of $\Psi_c$ with the inverse of the Dunkl map yields the required isomorphism $\Phi_c$. Note that when $c=0$ this then agrees  with the classical Harish-Chandra homomorphism \cite[\S 7]{EG}.

\begin{remark}
A similar result is obtained in \cite[Lemma 4.7]{KR}.
\end{remark}

\section{On the smooth locus of $\Lambda^s$}
\label{Micropi1}
\subsection{}
\label{GrojZ}
In this section we study the Lagrangian variety $Z$ in $\text{Hilb}_n(\mathbb C^2)$ introduced in $\S$\ref{almostcommuting}. Let us first give another description of $Z$, following \cite{Gr}. Let $\pi\colon \Hil \to S^n(\mathbb C^2)$ be the Hilbert-Chow morphism. It is known that $\pi$ is a resolution of singularities. Let $\Sigma = \mathbb C$, and view $\mathbb C^2$ as $T^*\Sigma$, a local model for a curve inside a surface. Since $\text{Hilb}^n(\Sigma)$ is isomorphic to the $n$-th symmetric product of $\Sigma$, we may view $\text{Hilb}^n(\Sigma)$ as either a subvariety of $\Hil$ or $S^n(\mathbb C^2)$, compatibly with the morphism $\pi$. The variety $S^n(\Sigma) = \mathbb C^n/S_n$ is stratified in an obvious way:

\[
S^n(\Sigma) = \bigsqcup_{\lambda \vdash n} S^n_\lambda(\Sigma),
\]
where the stratum $S^n_\lambda(\Sigma)$ contains those points of $S^n(\Sigma)$ whose multiplicities are given by $\lambda$, a partition of $n$. If we set $Z_\lambda^0 = \pi^{-1}(S_\lambda^n(\Sigma))$, then 
\[
Z = \bigsqcup_{\lambda \vdash n} Z_{\lambda}^0,
\]
is a closed Lagrangian subvariety of $\Hil$, with components $Z_\lambda = \overline{Z_{\lambda}^0}$. In terms of the analogy with category $\mathcal O$ for semisimple Lie algebras, $Z$ corresponds to the union of the conormal varieties to the $B$-orbits on the flag variety. 

$\text{Hilb}^n(T^*\Sigma)$ inherits a natural $\mathbb C^\times$-action from the $\mathbb C^\times$-action on the fibres of $T^*\Sigma$, and the components of $Z$ may also be described in terms of this action. Indeed each component of the fixed point set can naturally be identified with $S^n_\lambda(\Sigma)$ for a partition $\lambda$ of $n$, and if $\text{Fix}_\lambda$ denotes this component, we have
\[
\overline{Z_\lambda^0} = \overline{\{x \in \Hil: \lim_{t \to 0} t\cdot x \in \text{Fix}_\lambda\}},
\]
(here the right-hand side asserts both that the limit exists, and that it lies in $\text{Fix}_\lambda$). For more details see \cite[\S 7.2]{Na}

\subsection{}
\label{decompositionofM}
We now relate this to the almost commuting variety and the construction of $\Hil$ given in the previous section. The variety $\mathcal M_\text{nil}$ is described in \cite{GG} explicitly in terms of a stratification of $\mathfrak X$ which mimics Lusztig's stratification of a reductive group or Lie algebra. The strata are labelled by conjugacy classes of pairs $(\mathfrak l, \Omega)$ where $\mathfrak l$ is a Levi subalgebra and $\Omega$ is an orbit of $L$ (the associated Levi subgroup) on $\mathcal N_\mathfrak l \times V$ where $\mathcal N_{\mathfrak l}$ is the nilpotent cone of $\mathfrak l$ (there are only finitely many such orbits). A stratum $S$ is said to be \textit{relevant} if for some (and hence any) $(X,\ell) \in S$, the element $X$ is regular, and the subspace $\mathbb C[X]\ell$ has a $\mathbb C[X]$-invariant complement. Then we have the following result.

\begin{Theorem}\cite[\S 4.3]{GG}
The variety $\Lambda$ is given by
\[
\Lambda = \bigsqcup_{S \text{ relevant}} \overline{T^*_S(\mathfrak X)}.
\]
\end{Theorem}
Now $Z$ is obtained from $\Lambda$ as the quotient of $\Lambda^s$ by the $\text{PGL}(V)$-action, thus the components of $Z$ correspond to certain components of $\Lambda$. To make this correspondence explicit, we first give a natural labelling of the components of $\Lambda$. Recall that a bipartition of $n$ is an ordered pair of partitions $(\lambda, \mu)$ such that $|\lambda|+|\mu| = n$, where for a partition $\nu=(\nu_1\geq \nu_2 \geq \ldots \geq \nu_k)$ we write $|\nu|=\sum_{i=1}^k \nu_i$, the sum of the parts of $\nu$. 

\begin{lemma}
The relevant strata of $\mathfrak X$ and hence the components of $\Lambda$ are labelled by bipartitions of $n$.
\end{lemma}
\begin{proof}
Clearly we need only check that we may index relevant strata this way. If $(X,\ell) \in S$ where $S$ is a relevant stratum, then we may write $X = z+x$ the Jordan decomposition of $X$, where $z$ is semisimple and $x$ is nilpotent. The dimensions of the eigenspaces of $z$ give a partition $\nu$ of $n$ (this corresponds to the conjugacy class of the Levi subalgebra  $Z_\Lieg(z)$). Moreover, the condition that $(X,\ell)$ is relevant implies that if $v \in \ell$ is a nonzero vector, then its projection to an eigenspace of $z$ is either zero, or a cyclic vector for $x$ restricted to that subspace. Dividing $\nu$ according to this dichotomy we obtain a bipartition as required. 
\end{proof}

\begin{definition}
For a bipartition $(\lambda,\mu)$ we will write $\Lambda_{\lambda,\mu}$ for the component of $\Lambda$ corresponding to $(X,\ell)$ where the partition $\lambda$ measures the dimensions of the generalized eigenspaces of $X$ in which the projection of $\ell$ is nonzero, while $\mu$ measures the dimensions of the generalized eigenspaces in which the projection of $\ell$ is zero.
\end{definition}
To relate this to the variety $Z$ we need to examine which components of $\Lambda$ intersect the stable locus $\mathfrak M^s$. For a bipartition $(\lambda,\mu)$ of $n$ let $S_{\lambda,\mu}$ denote the corresponding relevant stratum and $\Lambda_{\lambda,\mu} = \overline{T^*_{S_{\lambda,\mu}}\mathfrak X}$ denote the conormal variety of the stratum labelled by $(\lambda,\mu)$ by the above convention.

We now note the following easy lemma, which is presumably well-known.
\begin{lemma}
\label{commuting}
Let $X \in \mathfrak{gl}(V)$ be a regular element. Then 
\[
Z_\Lieg(X) = \mathbb C[X],
\]
that is, the matrices which commute with $X$ are precisely the subalgebra of $\text{End}(V)$ generated by $X$. Moreover it follows that the elements of $Z_\Lieg(X)$ which are nilpotent form a linear subspace.
\end{lemma}
\begin{proof}
Let $V= \bigoplus_{i=1}^k V_i$ be the decomposition of $V$ into the generalized eigenspaces of $X$. On each $V_i$ we have $X_{|V_i} = \lambda_i + N_i$, where $\lambda_i$ is a scalar and $N_i$ is a regular nilpotent endomorphism (since $X$ is regular). The standard proof of the existence of a Jordan  decomposition (\text{e.g.} by the Chinese remainder theorem) shows that the projection operators $\pi_i$ from $V$ to the generalized eigenspaces of $X$ lie in $\mathbb C[X]$ (in fact they are polynomials in $X$ with no constant term). Thus $\iota_i \circ N_i\circ \pi_i$ also lies in $\mathbb C[X]$ for each $i$, ($1 \leq i \leq k$) where $\iota_i$ is the inclusion of $V_i$ into $V$
.
It is easy to check that the endomorphisms of a vector space $U$ which commute with a regular nilpotent $N$ are exactly the linear maps in the subalgebra $\mathbb C[N]$, and since $\{1,N,N^2,\ldots, N^{\dim(U)-1}\}$ clearly form a basis of $\mathbb C[N]$ this is a vector space of dimension $\dim(U)$. The nilpotent such endomorphisms are simply those in the span of $\{N,N^2,\ldots\}$. 
 
If $Y$ commutes with $X$, then clearly it preserves the eigenspaces of $X$, and so we may write it as $Y = \sum_{i=1}Y_i$, where $Y_i = \iota_i \circ y_i \circ \pi_i$ with $y_i \in \text{End}(V_i)$. Since $y_i$ clearly commutes with $X_{|V_i}$ and hence with $N_i$ we see $y_i \in \mathbb C[N_i]$ and so each $Y_i \in \mathbb C[X]$, and hence $Y \in \mathbb C[X]$. Moreover the nilpotent such $Y$ again clearly form a linear subspace as required.
\end{proof}

Now we have the following:

\begin{lemma}
\label{intersection}
The intersection $\mathfrak M^s \cap \Lambda_{\lambda,\mu}$ is dense in $\Lambda_{\lambda,\mu}$ if and only if $\mu=\emptyset$.
\end{lemma}
\begin{proof}
Let $S= S_{\lambda, \mu}$. We will in fact we show that 
\[
\mathfrak M^s \cap T^*_S(\mathfrak X) = \left\{\begin{array}{cc}\emptyset & \text{if } \mu \neq \emptyset, \\T^*_S(\mathfrak X) & \text{if } \mu = \emptyset. \end{array}\right.
\]
Suppose that $(X,Y,i,j)$ is a point in the intersection, and let $X = z+x$ be the Jordan decomposition of $X$, where $z$ is semisimple and $x$ is nilpotent. Let $V = \bigoplus_{i=1}^k V_i$ be the decomposition of $V$ according to the eigenspaces of $z$. The stability condition forces $j$ to be zero, so that $X$ and $Y$ commute, and hence $Y$ commutes with both $z$ and $x$. As $Y$ and $x$ commute with $z$, they both preserve the subspaces $V_i$ and we write $x_i$ for $x_{|V_i}$ and $y_i$ for $Y_{|V_i}$ respectively ($1 \leq i \leq k$). 

Since $(X,\mathbb C \cdot i)$ is relevant, $x_i$ must be a regular nilpotent, and so its centralizer in $\text{End}(V_i)$ is just $\mathbb C[x_i]$. Thus $y_i \in \mathbb C[x_i]$. Since the projections $V \to V_i$ also lie in $\mathbb C[X]$ it follows that $y_i$ (or rather its extension by zero to $V$) lies in $\mathbb C[X]$ and hence $Y \in \mathbb C[X]$. But then $\mathbb \langle X,Y \rangle = \mathbb C[X]$, and the stability condition then forces $i$ to be a cyclic vector for $X$, and hence $\mu=\emptyset$ as required. Conversely if $\mu = \emptyset$, then it is evident that any point in $T^*_S(\mathfrak X)$ satisfies the stability condition.
\end{proof}

\begin{remark}
\label{componentwarning}
Note that the above proof does \textit{not} imply that $\Lambda_{\lambda,\emptyset}\cap \mathfrak M^s$ is contained in $T_{S_{\lambda,\emptyset}}^*(\mathfrak X)$. For example when the bipartition is $((n),\emptyset)$, then the points of the corresponding stratum consist of matrices of the form $c + n$ where $c$ is a scalar and $n$ is a regular nilpotent, together with a line which is cyclic for $n$. The closure of the corresponding conormal variety contains stable quadruples $(X,Y,i,0)$ where $X$ is not regular. If we take $n=3$ say then it is easy to see that
\[
\big(\left(\begin{array}{ccc}0 & 0 & 0 \\0 & 0 & 1 \\0 & 0 & 0\end{array}\right), \left(\begin{array}{ccc}0 & 0 & 1 \\0 & 0 & 0 \\0 & 0 & 0\end{array}\right),\left(\begin{array}{c}0 \\0 \\1\end{array}\right),0\big),
\] 
is a stable quadruple, with neither $X$ or $Y$ regular. On the other hand, since $\text{U}$ is Zariski open, so is $(\text{U} \cap T^*\mathfrak G^\circ)/\mathbb C^\times$, thus if a component of $\Lambda$ intersects $\mathfrak M^s$ it will do so in a dense open subset. It follows that the Lemma tell us exactly which components of $\Lambda$ have nontrivial intersection with $\Lambda^s$.
\end{remark}

Recall we set $\Msn = \mathcal M^s\cap \mathcal M_\text{nil}$. It is $\text{GL}(V)$-stable. Its quotient by $\text{GL}(V)$ is the variety $Z$, and its quotient by the action of $\mathcal Z$ is the variety $\Lambda^s = \Lambda \cap \mathfrak M^s$. From Lemma \ref{intersection} and Remark \ref{componentwarning} we see that the components of $\Lambda^s$ are labelled by bipartitions of $n$ of the form $(\lambda, \emptyset)$. Since for a point $(X,Y,\ell) \in \Lambda^s$, the spectrum of $X$ yields the $n$-tuple of points in $S^n(\Sigma)$ under the composition of the quotient map with the Hilbert-Chow morphism, it is easy to check  that the component $Z_\lambda$ of $Z$ described in $\S$\ref{GrojZ} corresponds to the component of $\Lambda^s$ labelled by the bipartition $(\lambda, \emptyset)$. We will write $\Lambda^s_\lambda = \Lambda^s \cap \Lambda_{\lambda,\emptyset}$ for this component, and  $\Msn_\lambda$ for the corresponding components of $\Msn$.
 
\subsection{}

We now define a decomposition of $\Lambda^s $ into locally closed pieces. It will evidently be $G$-stable, and hence induce a decomposition of $Z$. In fact for convenience in this subsection we will work mostly with $\Msn$. Recall that if $(X,Y,i,j) \in \mathcal M^s$, then we have $j=0$, so that $X,Y$ commute with each other. We will thus prefer to write $(X,Y,i)$ for a point in $\Msn$ rather than $(X,Y,i,0)$.

\begin{definition}
Let $(X,Y,i) \in \Msn$. Let $X = z+x$ be the decomposition of $X$ into its semisimple and nilpotent parts respectively,  and let $V = \bigoplus_{i=1}^k V_i$ be the decomposition of $V$ into the eigenspaces of $z$. Since $Y$ commutes with $z$ it preserves each $V_i$, and its restriction to $V_i$ is a nilpotent endomorphism with Jordan type $\nu^{(i)}$ say. In this way we obtain a multipartition $\underline{\nu} = (\nu^{(1)},\nu^{(2)},\ldots,\nu^{(k)})$, and we may assume that $(\dim(V_1),\ldots,\dim(V_k))$ form a partition $\lambda$ of $n$. Let $\Msn_{\lambda,\underline{\nu}}$ denote the set of $(X,Y,i,j)\in \Lambda$ which give rise to the pair $(\lambda,\underline{\nu})$ in this way, where if $\lambda_i = \lambda_{i+1}$ then we must identify the parameters $(\lambda,\underline{\nu})$ and $(\lambda,\underline{\nu}')$ where $\underline \nu'$ is obtained from $\underline \nu$ by exchanging the $i$-th and $(i+1)$th partitions. We will also write $\Lambda^s_{\lambda,\underline{\nu}}$ for the corresponding subset of $\Lambda$.
\end{definition}
\noindent

We now wish to show that the pieces of our decomposition are smooth and to calculate their dimension. For this we begin with a technical lemma which gives a sort of normal form for elements of $\Msn$.

\begin{lemma}
Suppose that $(X,Y,v) \in \Msn$, and that $X$ and $Y$ are nilpotent endomorphisms. For $i,j \in \mathbb Z_{\geq 0}$, set $e_{i,j} = X^iY^j(v)$. Then we may find integers $k,s_0,s_2,\ldots s_{k-1},$ so that  $\{e_{ij}: 0 \leq i <k, 0\leq j < s_i\}$ is a basis of $V$, and moreover the action of $X$ on $V$ is given by
\[
X(e_{i,j}) = \left\{\begin{array}{cc}e_{i+1,j}, & \text{if } i<k-1, j < s_{i+1}, \\0, & \text{otherwise}.\end{array}\right.
\]
while if $\mathcal F = (F_i)$ denotes the decreasing filtration of $V$ given by $F_i = \text{im}(X^i)$, then action of $Y$ must satisfy:
\[
Y(e_{i,j}) = \left\{\begin{array}{cc}e_{i,j+1}, & \text{if } i<k, j< s_i-1 \\ w_{i+1}, & \text{if } j = s_i-1.\end{array}\right.
\]
where $w_{i+1}$ is some vector in $F_{i+1}$.

\end{lemma}
\begin{proof}
Let $k$ be the nilpotence of degree $X$, so that $X^k=0$ but $X^{k-1} \neq 0$, and set $F_i = \text{im}(X^i)$ so that $\mathcal F = (F_i)_{0 \leq i \leq k-1}$ is a $k$-step (decreasing) flag in $V$. Now since $Y$ commutes with $X$, and hence with every power of $X$ we immediately see that $Y$ preserves the flag $\mathcal F$, and clearly $X(F_i)= F_{i-1}$. Let $s_i = \dim(F_i/F_{i+1})$ for $0 \leq i <k$, clearly $\mu = (s_0,s_1,\ldots s_{k-1})$ is a partition of $n$ (the dual partition to the one given by the Jordan type of $X$). 

Now the vector $X^{k-1}(v)$ is clearly cyclic for the action of $\mathbb C[X,Y]$ restricted to $F_{k-1}$, but since $X$ acts by zero on $F_{k-1}$, it follows that $X^{k-1}(v)$ is cyclic for the action of $Y$ on $F_{k-1}$, and hence $\{e_{k-1,j}: 0 \leq j < s_{k-1}\}$ is a basis for $F_{k-1}$, and our description of the action of $X,Y$ on this part of the basis is established. But now considering $V/F_{k-1}$ and using induction on $k$ completes the proof.

\end{proof}

\begin{remark}
The $\text{GL}(V)$ orbit of $X$ is of course just the Richardson orbit attached to the parabolic determined by the flag $\mathcal F$. 
\end{remark}

\begin{lemma}
\label{weightingaction}
Suppose that $(X,Y,v) \in \Msn$ and $X,Y$ are nilpotent. Then given $a \geq b \in \mathbb Z$, we may find a homomorphism $\rho \colon \mathbb C^\times \to \text{GL}(V)$ such that 
\begin{enumerate}
\item
$\text{Ad}(\rho(t))(X) = t^a X$,
\item
 $\text{Ad}(\rho(t))(Y) = t^{b}Y_0 + O(t^{a})$, 
 \item
 $\rho(t)(v) = v$.
 \end{enumerate}
 Moreover, the triple $(X,Y_0,v)$ lies in $\Msn$.
\end{lemma}
\begin{proof}
We use the basis constructed in the previous Lemma. For $t \in \mathbb C^\times$, let $\rho(t)$ act on $V$ by 
\[
\rho(t)(e_{i,j}) = t^{-aj-bi}e_{i,j},  \quad (0 \leq i \leq k-1, 0 \leq j \leq s_i).
\]
Since $v= e_{0,0}$, condition $(3)$ is immediate, and similarly condition $(2)$ follows readily. Moroever to check condition $(1)$, we simply define $Y_0$ by
\[
Y_0(e_{i,j}) = \left\{\begin{array}{cc}e_{i,j+1}, & \text{if } j < s_i; \\0 & \text{if } j=s_i.\end{array}\right.
\]
The stability of the triple $(X,Y_0,v)$ is then clear.
\end{proof}

With these lemmas in hand, we can now check that the pieces of our decomposition are smooth. 

\begin{prop}
\label{pieces}
Each piece $\Msn_{\lambda,\underline{\nu}}$ is a connected smooth locally closed subvariety of $\Lambda^s$, and if $Z_{\lambda,\underline{\nu}}$ is its image under the quotient map to $Z \subset \Hil$, then $Z_{\lambda,\underline{\nu}}$ is an affine space bundle over $S_{\lambda}^n(\Sigma) \subset S^n(\mathbb C^2)$ via the Hilbert-Chow morphism. Moreover the dimension of the stratum $Z_{\lambda,\underline{\nu}}$ is given by
\[
n - \sum_{i=1}^k(\ell(\nu^{(i)})-1).
\]
\end{prop}
\begin{proof}
Since the strata $S^n_{\lambda}(\Sigma)$ are smooth locally closed subvarieties of $\mathbb C^n/S_n$ the last part of the Proposition implies the other claims. Thus it is enough to show that $Z_{\lambda,\underline{\nu}}$ is an affine space bundle over $S^n_\lambda(\Sigma)$. 

For this we first examine the case where $\lambda = (n)$, so that $S^n_\lambda(\Sigma)$ is just a point. The component $Z_{(n)}$ is the ``punctual'' Hilbert scheme, that is the moduli space of codimension $n$ ideals supported at $0 \in \mathbb C^2$, and in $T^*\mathfrak G^\circ$ it corresponds to the set
\[
\Msn_{(n)}= \{(X,Y,v) \in T^*\mathfrak G^\circ: [X,Y]=0, \mathbb C[X,Y]v = \mathbb C^n, X,Y \text{ nilpotent}\}
\]
For a point $(X,Y,v) \in \mathcal M^s$ we will write $[X,Y,v]$ for the corresponding point in the quotient $\text{Hilb}^n(T^*\Sigma)$. We claim that fixing the Jordan type of $X$ describes a locally closed affine space in $Z_{(n)}$.  To see this we use a torus action (for more details on the facts used here see for example \cite[Chapter 5]{Na}). Let $T^2 = (\mathbb C^\times)^2$ acts on $\Msn_{(n)}$ by 
\[
(t_1,t_2)(X,Y,v) = (t_1X,t_2Y,v).
\]
This action descends to the natural action of $T^2$ on $\text{Hilb}^n(T^*\Sigma)$ induced from the action of $T^2$ on $T^*\Sigma \cong \mathbb C^2$.

Recall the classical Bialynicki-Birula result which shows that a smooth projective variety with a $\mathbb G_m$-action has a natural decomposition into pieces which are affine bundles over the components of the fixed-point locus of that action. Now we cannot apply the Bialynicki-Birula result directly to $Z_{(n)}$, since it is not smooth. However the $T^2$ action on $Z_{(n)}$ extends to the whole Hilbert scheme $\text{Hilb}^n(\mathbb P^2)$ of $n$ points in the projective plane, which is smooth and projective, and hence the Bialynicki-Birula result may be applied to it. In fact, the action has finitely many fixed points, so that the pieces are affine spaces: if $x \in \text{Hilb}^n(\mathbb P^2)$ then the piece of the decomposition attached to $x$ is given by:
\[
P_x = \{y \in \text{Hilb}^n(\mathbb P^2): \lim_{t \to \infty} \nu(t)\cdot y = x\}.
\]
Choosing the $\mathbb G_m$-action to be given by a generic subgroup $\nu\colon \mathbb G_m \to T^2$ where $\nu(t) = (t^a,t^b)$ with $a>b>0$, we see by considering the action of $\mathbb G_m$ on the symmetric product $S^n(\mathbb C^2)$ (\textit{i.e.} on the spectra of matrices $X$ and $Y$) that no point $x$ outside $Z_{(n)}$ can have $\lim_{t \to \infty} \nu(t).x$ lying inside $Z_{(n)}$. Thus, since $Z_{(n)}$ is certainly $T^2$-invariant, it is a union of pieces of the decomposition of $\text{Hilb}^n(\mathbb P^2)$. (See for example \cite[\S 1]{ES} for the statement of the Bialynicki-Birula result and \cite[\S 2]{ES} for a similar use of the theorem for subvarieties of $\text{Hilb}^n(\mathbb P^2)$ which are preserved by the torus action). 

It is known (see \cite[\S 5.2]{Na} for a detailed discussion) that the $T^2$-action (on $\text{Hilb}^n(\mathbb A^2)$) has finitely many fixed points corresponding to monomial ideals of codimension $n$ all of which lie in $Z_{(n)}$, thus the pieces in the decomposition of $Z_{(n)}$ are naturally labelled by partitions of $n$. If $x_\lambda$ is the monomial ideal corresponding to the partition $\lambda$ we will write we write $P_{\lambda}$ for the corresponding piece. It is known that $P_\lambda$ is an affine space of dimension $n-\ell(\lambda)$. Letting $p \colon \mathcal M^s \to \Hil$ denote the quotient map, we claim that $p^{-1}(P_\lambda) = \Msn_{(n),\lambda}$.

To show this we use a one-parameter subgroup in $\text{GL}(V)$ as in Lemma \ref{weightingaction}. Since the $T^2$ action commutes with the $G$-action we have
\[
\begin{split}
\nu(t)[X,Y,\ell)] &= [t^{a}X, t^bY,\ell] \\
&= [(\rho(t^{-1})(t^{a}X),\rho(t^{-1})(t^bY),\rho(t^{-1})\ell)] \\
&= [X,Y_0 + Y_t,\ell],
\end{split}
\]
where $Y_t \in t^{-1}\text{End}(V)[t^{-1}]$. Thus since $\lim_{t \to \infty} (X,Y_0+Y_t,v) = (X,Y_0,v)$, which again by Lemma \ref{weightingaction} lies in $\Msn$, we see that the limit in $Z_{(n)}$ preserves the Jordan type of $X$. Since this also characterizes the fixed points of the $\mathbb C^\times$-action on $Z_{(n)}$ it follows immediately that $p^{-1}(P_{\lambda}) = \Msn_{(n),\lambda}$ as claimed.

In the general case, if $x \in S^n_\lambda(\Sigma)$ then $\pi^{-1}(x)$ is a product of punctual Hilbert schemes, indeed we have
\[
\pi^{-1}(x) \cong Z_{(\lambda_1)}\times Z_{(\lambda_2)} \times \ldots \times Z_{(\lambda_k)},
\]
where $\lambda = (\lambda_1,\lambda_2,\ldots,\lambda_k)$, and the isomorphism is given by taking, for $[X,Y,v] \in \pi^{-1}(x)$, the components of the nilpotent parts\footnote{Of course $Y$ is nilpotent, so is equal to its nilpotent part.} of $X$ and $Y$ in the eigenspaces of the semisimple part of $X$, along with the projection of the vector $v$ to that eigenspace. It follows immediately that the intersection of the $Z_{\lambda, \underline{\nu}}$ with $\pi^{-1}(x)$ is a product of affine spaces, and hence itself an affine space. Moreover, the dimension of the affine space is clearly
\[
\sum_{i=1}^k (\lambda_i - \ell(\nu^{(i)})) = n - \sum_{i=1}^k \ell(\nu^{(i)}). 
\]
But since $S^n_\lambda(\Sigma)$ clearly has dimension $k=\ell(\lambda)$, it follows that
\[
\dim(Z_{\lambda,\underline{\nu}})= n - \sum_{i=1}^k(\ell(\nu^{(i)})-1).
\]
as claimed. 
\end{proof}

\begin{remark}
Notice that it follows from the dimension formula that the pieces of the decomposition of maximal dimension are the pieces $Z_{\lambda,\underline{\nu}}$ where 
\[
\underline{\nu} = ((\lambda_1), (\lambda_2),\ldots, (\lambda_k)).
\]
Notice moreover that the pieces of our decomposition which have dimension $n-1$ are  those for which exactly one partition $\nu^{(i)}$ has two parts. From this it follows that such a piece lies in exactly two components of $Z$, those being $Z_\lambda$ and $Z_{\lambda'}$ where $\lambda'$ is obtained from $\lambda$ by replacing $\lambda_i$ by the two parts of $\nu^{(i)}$. Thus the decomposition of $Z$ also allows us to describe which components of $Z$ intersect in codimension $1$. 

For example, the components which intersect the punctual Hilbert scheme $Z_{(n)}$ in codimension $1$ are those of the form $Z_{(k,n-k)}$, while the only component intersecting $Z_{(1^n)}$ in codimension $1$ is $Z_{(2,1^{n-2})}$.
\end{remark}

\subsection{}
\label{smoothpieces} 
We now use our decomposition of $\Msn$ to find the smooth locus of $\Msn$. We claim that its components are exactly the maximal dimensional pieces of our decomposition. Thus we set $\Msnt_\lambda = \Msn_{\lambda,\underline{\nu}}$ where $\nu = ((\lambda_1),(\lambda_2),\ldots,(\lambda_k))$, that is, each partition in $(\underline\nu)$ has a single part. By the dimension formula in Proposition \ref{pieces} each $\Msnt_\lambda$ has dimension equal to $\dim(\Lambda^s)$. Let $\tilde{Z}_\lambda$ and $\tL_\lambda$ denote the corresponding pieces of $Z$ and $\Lambda^s$ respectively.

\begin{lemma}
\label{smoothlocus}
Let $\Msnt_\lambda$ be as above. Then $\Msnt_\lambda$ is an open dense subvariety of $\Msn_\lambda$ consisting of smooth points of $\Msn$. 
\end{lemma}
\begin{proof}
Since $\Msnt_\lambda$ is a piece of our decomposition, we have already established that it is smooth. 
To conclude that $\Msnt_\lambda$ consists of smooth points of $\Msn$ it remains to check that $\Msnt_\lambda$ does not intersect any other component of $\Msn$. For this suppose that $(X_0,Y_0,v_0) \in \Msnt_\lambda$ lies in some other component of $\Msn$, say $\Msn_\mu$. Then since $\Msnt_\mu$ is dense in $\Msn_\mu$ we see that $(X_0,Y_0,v_0)$ must lie in the closure of $\Msnt_\mu$. Thus suppose that $(X_i,Y_i,v_i)_{i \geq 1} \in \Msnt_\mu$ converges\footnote{For this argument, we use the complex topology on our varieties, not the Zariski topology.} to $(X_0,Y_0, v_0)$. But now the Jordan type of each $Y_i$ is given by $\mu$, and since the $(Y_i)$ converges to $Y_0$ we must have $\mu \geq \lambda$ in the dominance ordering on partitions, and hence certainly we must have $\ell(\mu) \leq \ell(\lambda)$. 

On the other hand, letting $s \colon \Msn \to \mathfrak h/S_n$ be the map given by taking the eigenvalues of $X$. Since the sequence $s(X_i,Y_i,v_i)$ consists of $\ell(\mu)$ distinct points with multiplicities $\mu_i$, $(1 \leq i \leq \ell(\mu))$, it follows that either $X_0$ has at most $\ell(\mu)$ distinct eigenvalues and so $\ell(\mu) \geq \ell(\lambda)$. Hence we see that we must have $\ell(\mu) = \ell(\lambda)$, and then moreover the eigenvalue multiplicities of $X_0$ must be equal to those of the $X_i$'s (as there can be no ''collision'' of eigenvalues), whence we have $\lambda = \mu$ as desired.
\end{proof}

\begin{remark}
It is easy to see that every other piece in our decomposition lies in more than one component of $Z$, and hence the union of our subsets $\tilde{Z}_\lambda$ in fact yields the entire smooth locus of $Z$.
\end{remark}

\begin{example}
If $\lambda = (n)$, then $Y$ has a single Jordan block, and the space $\tL_\lambda$ is just an affine space bundle over $\mathcal O_Y$ with fibres of dimension $2n-1$. The corresponding subset in $Z_{(n)}$ is an affine space of dimension $n-1$, it is the open cell in an affine paving of $Z_{(n)}$.
\end{example}

 We now wish to compute the fundamental group of the smooth loci $\tilde{\mathcal M}^{\text{sn}}_\lambda$. For this we make the following definition.

\begin{definition}
Let $\lambda$ be a partition of $n$. We may write $\lambda$ uniquely as $(i_1^{c_1},i_2^{c_2},\ldots, i_r^{c_r})$ where $1 \leq i_1<i_2< \ldots < i_r \leq n$. Let $k = 
\ell(\lambda) = \sum_{i=1}^r c_i$. Let $\mathcal B_\lambda$ be the fundamental group of the configuration space of $k$ labelled points in $\mathbb C$ where the labels lie in $\{i_1,i_2,\ldots, i_r\}$ and there are $c_j$ points with label $i_j$. In our earlier notation this is just the space $S^n_\lambda(\mathbb C)$. Let $\mathcal P_\lambda$ be the subgroup of $\mathcal B_\lambda$ corresponding to the cover of the configuration space where all $k$ points are distinct, \textit{i.e.} the pure braid group on $k$ strands, and let $\Sigma_\lambda$ be the quotient $\mathcal B_\lambda/\mathcal P_\lambda$ so that $\Sigma_\lambda \cong S_{a_1}\times S_{a_2} \times \ldots \times S_{a_k}$, where $S_a$ denotes the symmetric group on $a$ letters.
\end{definition}
Thus for example if $\lambda = (1^n)$ this group is the standard braid group on $n$ strands.
Note that $\mathcal B_\lambda$ is a subgroup of the braid group on $k$ strands. It contains the parabolic braid subgroup given by the partition associated to the composition $(c_1,c_2,\ldots,c_r)$, but is strictly bigger -- for example it contains the isotopies which rotate points labeled $i$ and $j$ a full $2\pi$ about each other.  

The fundamental group of $\tilde{\mathcal M}^{\text{sn}}_\lambda$ is closely related to $\mathcal B_\lambda$ as we now show. For $\mathbf a = (a_1,a_2,\ldots a_{r-1}) \in \mathbb C^{r-1}$, let 
\[
J_\mathbf a(t) = \left(\begin{array}{cccccc}t & a_1 & a_2 & \ldots & a_{r-2} & a_{r-1} \\0 & t & a_1 &  &  & a_{r-2} \\0 & 0 & \ddots & \ddots & \ddots & \vdots \\0 & 0 & 0 & \ddots & a_1 & a_2 \\0 & 0 & 0 & 0 & t & a_1 \\0 & 0 & 0 & 0 & 0 & t\end{array}\right) \in \text{Mat}_r(\mathbb C).
\]

Let $\mu = \{\mu_1,\ldots,\mu_k\} \in S_{\lambda}^n(\mathbb C)$ with labels (\textit{i.e.} multiplicities) given by the parts of $\lambda$ written in increasing order: $d_1\leq d_2 \ldots \leq d_k$ so that $d_j \in \{i_1,i_2,\ldots,i_r\}$ for each $j$, ($1 \leq j\leq k$). Then if $\mathbf a^j \in \mathbb C^{d_j-1}$ for $1 \leq j \leq k$, set $X_\lambda(\mathbf a^1,\mathbf a^2,\ldots,\mathbf a^k)$ to be the block diagonal matrix:
\[
\left(\begin{array}{cccc}J_{\mathbf a^1}(\mu_1) & 0 & 0 & 0 \\0 & J_{\mathbf a^2}(\mu_2) & 0 & 0 \\0 & 0 & \ddots & 0 \\0 & 0 & 0 & J_{\mathbf a^k}(\mu_k)\end{array}\right)
\]
and similarly set $Y_\lambda$ to be the block diagonal nilpotent matrix with Jordan blocks of size $d_1,d_2,\ldots,d_k$. Finally, set $v_\lambda \in \mathbb C^n$ to be the vector with $v_j = 1$ if $j \in \{\sum_{s=1}^t d_s: 1\leq t \leq k\}$ and $v_j=0$ otherwise. It is then easy to check that $(X_\lambda(\mu,\mathbf a^1,\ldots, \mathbf a^k), Y_\lambda,v_\lambda) \in \tilde{\mathcal M}^{\text{sn}}_{\lambda}$, and moreover if $(X_\lambda(\mu, \mathbf a^1,\ldots,\mathbf a^k),Y_\lambda,v_\lambda)$ and $(X_\lambda(\mu',\mathbf b^1,\ldots,\mathbf b^k),Y_\lambda, v_\lambda)$ are two such points then they lie in the same $\text{GL}(V)$-orbit if and only if they are equal up to permutation of the Jordan blocks of the $X$s -- \text{i.e.} conjugate by an element $\Sigma_\lambda$ viewed as the subgroup of the group of permutation matrices which interchanges our diagonal blocks of the same size (respecting the order within each block). Thus if we set $\mathcal S_\lambda$ to be the subspace of $\tilde{\mathcal M}^{\text{sn}}_\lambda$ consisting of points of the form $(X(\mu,\mathbf a^1,\ldots, \mathbf a^k),Y_\lambda,v_\lambda)$, then we have a natural map
\[
p_\lambda \colon \text{GL}(V)\times \mathcal S_\lambda \to \tilde{\mathcal M}^{\text{sn}}_\lambda,
\]
given by the $\text{GL}(V)$-action, and this map is readily seen to be a $\Sigma_\lambda$-covering, where $\Sigma_\lambda$ acts on $\text{GL}(V)$ on the right (embedded as a subgroup of permutation matrices) and on $\mathcal S_\lambda$ by the restriction of the $\text{GL}(V)$-action on $\Msn$.
 
\begin{prop}
\label{microlocalpi1}
Let $\lambda$ be a partition of $n$. Then we have exact sequences:

\xymatrix{
& & &  1\ar[r] & \mathbb Z \ar[r] & \pi_1(\tilde{\mathcal M}^{\text{sn}}_\lambda) \ar[r] & \mathcal B_\lambda \ar[r] & 1
}

\xymatrix{
& & &  1\ar[r] & \mathbb Z\times \mathcal P_\lambda \ar[r] & \pi_1(\tilde{\mathcal M}^{\text{sn}}_\lambda) \ar[r] & \Sigma_\lambda \ar[r] & 1
}
\noindent
Moreover, $\mathcal P_\lambda$ is normal in $\pi_1(\tilde{\mathcal M}^{\text{sn}}_\lambda)$, the subgroup $\mathbb Z$ is central, and $\pi_1(\tilde{Z}_\lambda) \cong \mathcal B_\lambda$.
\end{prop}
\begin{proof}
First notice that the variety $\tilde{Z}_{\lambda}$ is homotopy equivalent to $S_\lambda^n(\Sigma)$ since it is an affine space bundle over it, thus clearly $\pi_1(\tilde{Z}_\lambda) \cong \mathcal B_\lambda$. Now since the configuration space of $n$ distinct points in the complex plane is known to be a $K(\pi,1)$ (see for example \cite{D}), it follows that $S^n_{\lambda}(\Sigma)$ is a $K(\pi,1)$. The long exact sequence of a fibration then immediately shows that the fundamental group of $\tilde{\mathcal M}^{\text{sn}}_\lambda$ is an extension of $\mathcal B_\lambda$ by $\pi_1(\text{GL}(V))\cong \mathbb Z$. This yields the first of our exact sequences. 

Now consider the covering $p_\lambda$.  Since the total space is a product, the theory of covering spaces yields the second exact sequence. Moreover, the regular covering of $\tilde{\mathcal M}^{\text{sn}}_\lambda$ given by taking the universal cover of $\text{GL}(V)$ over the total space of $p_\lambda$ shows that $\mathcal P_\lambda$ is a normal subgroup of $\pi_1(\tilde{\mathcal M}^{\text{sn}}_\lambda)$. Thus since $\mathbb Z$ is clearly centralised by $\mathcal P_\lambda$ it suffices to show that $\mathbb Z$ is central in the quotient by $\mathcal P_\lambda$, which is isomorphic to the fundamental group of $\text{GL}(V)/\Sigma_\lambda$. But this is equivalent to showing that the action of $\Sigma_\lambda$ on $\pi_1(\text{GL}(V))$ is trivial which is immediate, because the action of $\Sigma_\lambda$ is the restriction of the action of the connected group $\text{GL}(V)$.

\end{proof}

\begin{remark}
\label{quotientofpi1}
Note that one can also describe the fundamental group of $\text{GL}(V)/\Sigma_\lambda$ explicitly, for example by considering it as a subgroup of the universal cover $\tilde{G}$ of $\text{GL}(V)$. Indeed $\tilde{G} = \{(g,t) \in \text{GL}(V)\times \mathbb C: \det(g) = \text{exp}(2\pi i t)\}$, and we may realise 
\[
\pi_1(\text{GL}(V)/\Sigma_\lambda) = \{(w,n) \in \Sigma_\lambda \times (1/2)\mathbb Z: \text{sgn}(w) = \text{exp}(2\pi i n)\}
\]
where $\text{sgn}$ is the usual sign function on $S_n$ restricted to $\Sigma_\lambda$. Here the subgroup $\mathbb Z\cong \pi_1(\text{GL}(V))$ is realized as $\{(1,n): n \in \mathbb Z\}$. For use in the next section, we also note here that, if we fix a generator $\gamma$ of $\pi_1(\text{GL}(V))$ and take any $q \in \mathbb C^\times$, say $q = \text{exp}(2\pi i \alpha)$, then $\pi_1(\tilde{\mathcal M}^{\text{sn}}_\lambda)$ has a one-dimensional representation $L_q$ on which $\gamma$ acts by $q$. Indeed the quotient $\pi_1(\text{GL}(V)/\Sigma_\lambda)$ obviously has such a representation, by taking the restriction of the representation $\Sigma_\lambda \times (1/2)\mathbb Z$ which sends $(w,n) \mapsto \text{exp}(2\pi i n\alpha)$. It is easy to check that in the case $\lambda = (1^n)$ the representations $L_q$ are given by the monodromy of the $\mathcal D$-module $s^c$, where $s$ is as in $\S$\ref{deformedHC}. We will use this fact in comparing our microlocal $KZ$-functors to the original $KZ$ functor.
\end{remark}

\section{Microlocal KZ functors}
\label{MKZfunctors}
\subsection{}
\label{categoryC_c}
Let $\mathcal O_c^{\text{sph}}$ be the category of representations of the spherical rational Cherednik algebra corresponding to the category $\mathcal O_c$ of modules for the rational Cherednik algebra, in other word, $\mathcal O_c^{\text{sph}}$ is the category of $e\mathcal H_c e$-modules on which $\mathbb C[\mathfrak h^*]^W_+$, the augmentation ideal of $\mathbb C[\mathfrak h^*]^W$, acts locally nilpotently. In this section we use our analysis of the variety $\Lambda$ to study $\mathcal O_c^{\text{sph}}$. 

It is well known that projective spaces are $\mathcal D$-affine. More precisely, if we assume\footnote{The $\frac{1}{n}$ factor comes from the fact that we use $\text{tr}_{|\text{Lie}(\mathcal Z)}$ as a basis vector for the dual space of $\text{Lie}(\mathcal Z)$.} that $c \notin \frac{1}{n}\mathbb Z_{<0}$ then the localization functor gives an equivalence between $\mathcal D_{\mathfrak X,c}$-modules and $\mathcal D_c(\mathfrak X)$-modules. Next recall the adjoint $\text{Loc}$ of the quantum Hamiltonian reduction functor, (in \cite{GG} the corresponding functor is denoted $^\top\mathbb H$). Let $\mathcal L_c$ be the  $\mathcal D_{\mathfrak X,c}$-module whose global sections are:
\[
\mathcal D_c(\mathfrak X)/\mathcal D_c(\mathfrak X)\cdot \mathfrak g_c,
\]
where $\mathfrak g_c$ is the image of $\mathfrak{sl}_n$ under the twisted quantized moment map $\tau_c$ associated to the action of $\text{PGL}(V)$ on $\mathfrak X$ via the isomorphism $\mathfrak{sl}_n \hookrightarrow \mathfrak{gl}_n \to \mathfrak{pgl}_n$. The twisted Harish-Chandra homomorphism $\Psi_c$ defined in Section \ref{deformed} yields an isomorphism between $e\mathcal H_c e$ and $\text{End}(\mathcal L_c)^{\text{op}}$, so that given a $e\mathcal H_c e$-module $M$ we may associate to it a $\mathcal D_{\mathfrak X,c}$-module
\[
\text{Loc}(M) = \mathcal L_c \otimes_{e\mathcal H_c e} M.
\]
This is the left adjoint to the functor $\mathbb H\colon \mathcal C_c \to \mathcal O_c^{\text{sph}}$ which sends $\mathcal F \mapsto \Gamma(\mathcal F)^{\text{SL}(V)}$, and it is known that the adjunction morphism $\mathbb H\circ \text{Loc}(M) \to M$ is an isomorphism for all $M \in \text{ob}(\mathcal O^{\text{sph}}_c)$. Recall that since $\text{SL}(V)$ commutes with the action of $\mathcal Z \subset \text{GL}(V)$ on $\mathfrak G^\circ$, we may consider the category of $\text{SL}(V)$-equivariant $D_{\mathfrak X,c}$-modules (see Appendix \ref{twistingstuff} for details and references). In \cite{GG}, Gan and Ginzburg establish the following result: let $\mathcal C_c$ be the full subcategory of the category of $(\mathcal D_{\mathfrak X,c}, \text{SL}(V))$-modules whose objects are have their characteristic cycle contained in $\Lambda$. 

\begin{Theorem}\cite[\S 7]{GG}. 
\label{GGtheorem}
Suppose that $c \notin \frac{1}{n}\mathbb Z_{<0}$. Then the functor $\text{Loc}$ yields an equivalence of categories between $\mathcal O_c^{\text{sph}}$ and $\mathcal C_c/\text{ker}(\mathbb H)$.
\end{Theorem}

\begin{remark}
Here $\text{ker}(\mathbb H)$ is the subcategory of objects in $\mathcal C_c$ which are annihilated by $\mathbb H$. Since $\text{SL}(V)$ is reductive, the functor of Hamiltonian reduction $\mathbb H$ is exact, so that $\text{ker}(\mathbb H)$ is a Serre subcategory. Kashiwara and Rouquier use the essentially same procedure, but then descend to the Hilbert scheme. 
\end{remark}

\subsection{}
We now wish to construct our microlocal analogues of the $KZ$-functors. For this we need the language of microdifferential operators. The most ``algebraic'' approach is to work with the ring $\widehat{\mathcal E}_X$ of formal microdifferential operators (see \cite[Chapter 7]{Kbook} where it is constructed using the calculus of total symbols) however there are a number of variants: the ring $\mathcal E_X$ of microdifferential operators with growth conditions, the sheaf $\mathcal E_X^{\infty}$ of infinite order operators, and the sheaf $\mathcal E_X^\mathbb R$, see for example \cite{K86} for a survey.

We briefly recall some basic properties (in the formal setting). Given $X$ a complex manifold, $\E_X$ is a $\mathbb Z$-filtered sheaf of rings on $T^*X$. If $\pi\colon T^*X \to X$ is the bundle map then $\E_X$ contains the sheaf $\pi^{-1}\mathcal D_X$ as a subring, and the quotients $\E_X(m)/\E_X(m-1) \cong \mathcal O_{T^*X}(m)$ where $\mathcal O_{T^*X}(m)$ denotes the sheaf of holomorphic functions on $T^*X$ homogeneous of degree $m$ along the fibres of $T^*X$. It is also known that $\E_X$ is flat over $\pi^{-1}\mathcal D_X$, and one can describe the characteristic cycle of a $\mathcal D_X$-module in terms of the associated $\E_X$-module: indeed for a coherent $\mathcal D_X$-module $\mathcal F$ we have: 
\begin{equation}
\label{EmodSS}
\text{Supp}(\E_X\otimes_{\pi^{-1}\mathcal D_X} \pi^{-1}\mathcal F) = \text{SS}(\mathcal F),
\end{equation}
where $\text{SS}$ denotes singular support (the multiplicities given by the characteristic cycle can also be recovered as discussed below). Note that when working with coherent $\mathcal D_X$-modules, we often consider a good filtration. In the context of $\E_X$-modules this corresponds to taking an $\E_X(0)$-lattice.

Suppose that $\mathcal F$ is a coherent $\mathcal D_X$-module, and $U$ is an open subset of $X$. It is well-known that if $\text{CC}(\mathcal F)_{|T^*U} \subseteq U$ then $\mathcal F_{|U}$ is in fact $\mathcal O_{U}$-coherent, and hence a vector bundle with connection, or local system, on $U$. The theory of regular holonomic $\E_X$-modules shows that this local constancy has a natural generalization to other Lagrangian submanifolds of $T^*X$. Let $\mathring{T}^*X$ denote the complement to the zero section in $T^*X$. If $\Omega$ is an open subset of $T^*X$, then we define a functor $\mu_\Omega$ from $\mathcal D_X$-modules to $\E_{X|\Omega}$-modules by setting $\mu_\Omega(\mathcal F) = (\E_X\otimes_{\pi^{-1}\mathcal D_X}\pi^{-1}(\mathcal F))_{|\Omega}$.

\begin{Theorem}
\label{Morselocalsystem}
Let $\Omega$ be an open subset of $\To^*X$ and let $\Lambda$ be a homogeneous smooth Lagrangian submanifold in $\Omega$ such that the projection from $\Lambda$ to $X$ has constant rank. Then there is an exact functor $\Phi_\Lambda$ from the category of holonomic $\E_X$-modules $\mathcal N$ with $\text{SS}(\mathcal N) \cap \Omega \subset \Lambda$ to the category of local systems on $\Lambda$.
\end{Theorem}
\begin{proof}
This can be seen topologically or analytically. One approach works with the ring $\mathcal E_X^\mathbb R = \mathcal{H}^n(\mu_{\Delta}(\mathcal O_{X\times X}^{0,n)}))$, where $\mu_{\Delta}$ is the microlocalization functor, $n=\dim(X)$, and $\mathcal O_{X\times X}^{(0,n)}$ is the sheaf of holomorphic forms on $X\times X$ which are $n$-forms with respect to the second variable. This is a sheaf of rings which contains $\mathcal E_X$ and is faithfully flat over it.

Now by the assumption on the projection from $\Lambda$ to $X$, we may write $\Lambda = T^*_ZX\cap \Omega$ for some submanifold $Z$ of $X$. Then we have the simple holonomic system $\mathcal C_{Z|X} = \mathcal E_X\otimes_{\mathcal D_X} \mathcal B_{Z|X}$ along $\Lambda$. Now for any holonomic $\mathcal E_X$-module $\mathcal N$ by\footnote{The proof of Theorem $1.3.1$ in that paper is given in \cite{K83}.} \cite[Theorem 1.3.1(i)]{KK1} we have $\mathcal Ext^j_{\mathcal E_X}(\mathcal C_{Z|X},\mathcal N^\mathbb R)=0$ for $j \neq 0$ and  moreover $\mathcal Hom_{\mathcal E_X}(\mathcal C_{Z|X},\mathcal N^\mathbb R)$ is a locally constant  sheaf on $\Lambda$. Thus the map
\[
\mathcal M \mapsto \mathcal Hom_{\mathcal E_X}(\mathcal C_{Z|X}, \mathcal N^{\mathbb R}),
\]
where $\mathcal N = (\mathcal E_X\otimes_{\mathcal D_X} \mathcal M)_{|\Omega}$, yields an exact functor to local systems. 

Alternatively, via the Riemann-Hilbert correspondence, we may work topologically with the perverse sheaf $P$ determined by $\mathcal M$. In this situation the functor to local systems is the vanishing cycles functor, and its construction is described in \cite[\S 5]{MV1}. A more detailed discussion, giving a number of different approaches in the topological setting, is given in \cite[\S 4]{GMV}.
\end{proof}

\begin{remark}
Note that \cite[Theorem 1.3.1(i)]{KK1} also states that there is a canonical isomorphism:
\[
\mathcal C_{Z|X}^\mathbb R \otimes_\mathbb C \mathcal{H}om_{\mathcal E_X}(\mathcal C_{Z|X}, \mathcal N^\mathbb R) \to \mathcal N^\mathbb R.
\]
where $\mathcal N^{\mathbb R}$ denotes the extension of $\mathcal N$ to $\mathcal E^{\mathbb R}_X$. If $\Lambda =\gamma^{-1}\gamma(\Lambda)$ where $\gamma \colon \To X \to \mathbb P^*X$ is the map to the projectivised cotangent bundle, then by \cite[Theorem 1.3.1 (ii)]{KK1} the sheaf $\mathcal M^\infty$ (the extension of $\mathcal M$ to the ring $\mathcal E^{\infty}_X$ of infinite order microdifferential operators) is completely determined by $\mathcal M^\mathbb R$, and if $\mathcal M$ has regular singularities, then by \cite[Theorem 5.2.1]{KK1} it is determined by $\mathcal M^\infty$. Thus in this case there is in fact an equivalence between regular singularities $\mathcal E_{X|\Omega}$-modules supported on $\Lambda$ and local systems on $\Lambda$. In general however, regular singularities $\mathcal E_X$-modules supported on a smooth homogeneous Lagrangian are equivalent to twisted local systems on $\Lambda$ (see \cite[\S 10]{K86} for more details).

Finally, using for example the explicit description of the functor $\Phi_\Lambda$ as homomorphisms from $\mathcal C_{Z|X}$ we see that if our holonomic modules are $(H,\lambda)$-equivariant for some group $H$, then $\Phi_\Lambda$ restricts to a functor taking values in the category of $(H,\lambda)$-twisted local systems on $\Lambda$.
\end{remark}

\begin{remark}
The condition that $U$ is an open in $\mathring{T}^*X$ is not particularly restrictive: if we start with a $\E_X$-module $\mathcal F$ we may replace $X$ by $Z= X\times \mathbb C$ and $\mathcal F$ by $\mathcal F \boxtimes \delta_{0}$ where $\delta_0$ denotes the module generated by the $\delta$-function at $0 \in \mathbb C$. Since the conormal bundle of $X\times \{0\}$ in $Z$ is now $X \times T_0^*\mathbb C$ we have effectively moved the support of $\mathcal F$ off the zero section.
\end{remark}

\subsection{}
\label{mKZdef}
By the discussion in $\S$\ref{DmodsonG} we may view the category $\mathcal C_c$ of $\S$\ref{categoryC_c} as the abelian subcategory of $(\GL(V),c.\text{tr})$-equivariant module on $\mathfrak G^\circ$ whose singular support is contained in $\Msn$ (the preimage of $\Lambda$). Given a component $\Msn_\lambda$ of $\Msn$ we may choose an open set $\Omega_\lambda$ in $T^*\mathfrak G^\circ$ such that $\Omega_\lambda\cap \Msn = \tilde{\mathcal M}^{\text{sn}}_\lambda$.
From our explicit description of $\tilde{\mathcal M}^{\text{sn}}_\lambda$ in \S\ref{smoothpieces} it immediately follows that the projection from $\tilde{\mathcal M}^{\text{sn}}_\lambda$ to $\mathfrak G^\circ$ has constant rank (indeed the fibre of the projection from $\To\mathfrak G^\circ$ over a point $(X,v) \in \mathfrak G^\circ$ is just the space of nilpotent matrices which are regular in each generalized eigenspace of $X$, thus in fact the restriction of the projection is a fibre bundle with smooth fibres). Thus by composing the functor of Theorem \ref{Morselocalsystem} with the functor $\mu_{\Omega_\lambda}$, we obtain an exact functor
\[
\mathcal{KZ}_\lambda \colon \mathcal C_c \to \mathcal{LS}(\tilde{\mathcal M}^{\text{sn}}_\lambda).
\]
where $\mathcal{LS}(\tilde{\mathcal M}^{\text{sn}}_\lambda)$ denotes the category of local systems on $\tilde{\mathcal M}^{\text{sn}}_\lambda$. Note that  the Lagrangian $\tilde{\mathcal M}^\text{sn}_{(1^n)}$ is a subset of $\mathfrak G^\circ$, so that in that case we only need the classical equivalence between vector bundles with flat connection and local systems.  

\begin{definition}
By choosing a base point in $\tilde{\mathcal M}^{\text{sn}}_\lambda$ we may think of $\mathcal{KZ}_\lambda$ as a functor to the category $\text{Rep}(\pi_1(\tilde{\mathcal M}^{\text{sn}}_\lambda))$ of finite-dimensional representations of $\pi_1(\tilde{\mathcal M}^{\text{sn}}_\lambda)$. Moreover, since we are working with $(\text{GL}(V),c.\text{tr})$-equivariant modules $\mathcal F$, the local system $\mathcal{KZ}_\lambda(\mathcal F)$ must also be. Now by Proposition \ref{microlocalpi1} we see that $\pi_1(\tilde{\mathcal M}^\text{sn}_\lambda)$ is an extension of $\mathcal B_\lambda$ by $\pi_1(\text{GL}(V)) \cong \mathbb Z$. The twisted equivariance shows that the subgroup $\mathbb Z = \pi_1(\text{GL}(V))$ must act by $q=\text{exp}(2\pi i nc)$. Thus it follows that the tensor product $L_{q^{-1}}\otimes \mathcal{KZ}_\lambda(\mathcal F)$ descends to a representation of $\mathcal B_\lambda$, where $L_{q^{-1}}$ is the one-dimensional representation constructed in Remark \ref{quotientofpi1}. Let $KZ_\lambda(\mathcal F)$ denote  this representation of $\mathcal B_\lambda$, an object of the category of finite-dimensional $\mathcal B_\lambda$-representations $\text{Rep}(\mathcal B_\lambda)$. By composing with the functor $\text{Loc}$ we also obtain a functor from $\mathcal O^{\text{sph}}_c$ to $\text{Rep}(\mathcal B_\lambda)$. Since it should be clear from context, we will abuse notation slightly and again write $KZ_\lambda$ for this composition of functors\footnote{Of course one also obtains a functor from $\mathcal O_c$ to $\text{Rep}(\mathcal B_\lambda)$ by composing with the functor from $\mathcal O_c$ to $\mathcal O^\text{sph}_c$ sending $M \mapsto eM$.}. 
\end{definition}

\begin{remark}
Since the representation $L_{q^{-1}}$ is not canonically defined, it is perhaps more natural to have the functors $KZ_{\lambda}$ take values in the category of twisted local systems. However, in order to connect more directly with the original $KZ$ functor we prefer here to descend to ordinary local systems.
\end{remark}

For use in the next section, we also wish to note the connection between our functors and the characteristic cycles for modules in $\mathcal O^\text{sph}_c$ defined by Gan-Ginzburg \cite[\S 7.5]{GG} and Gordon-Stafford \cite[\S 2.7]{GS} (which are shown to be equivalent in \cite[\S 1.10]{GGS}). If $M$ is a module in $\mathcal O^\text{sph}_c$, then the definition in \cite{GG} uses the intersection of $\text{CC}(\text{Loc}(M))$ with the stable locus $\text{U}$, and takes its quotient by $\GL(V)$ to obtain Lagrangian cycle in $\text{Hilb}^n(\mathbb C^2)$. By equation (\ref{EmodSS}), or more precisely by its refinement which counts multiplicities, this characteristic cycle is obtained by taking the dimension of our local systems $KZ_\lambda(M)$. 

\subsection{}
\label{microsupportH}

We now relate our functors to the results in \cite{KR}. Recall\footnote{We refer the reader to their paper for details and terminology such as $\mathscr W$-algebras and $F$-structures, but note we summarise their work using our notation not theirs.} that they construct a $\mathscr W$-algebra on $\text{Hilb}^n(\mathbb C^2)$ by ``symplectic reduction'' (see also \ref{Walgconstruction}). Starting with the $\mathcal D$-module $\mathcal L_c$ of $\S$\ref{categoryC_c}, they extend coefficients to obtain a module (also denoted $\mathcal L_c$) for $\mathscr W_{T^*\mathfrak G}$ the standard $\mathscr W$-algebra on $T^*\mathfrak G$, which is $(\GL(v),c.\text{tr})$-equivariant and is supported on $\mathcal M$. Since $\GL(V)$ acts freely on $\text{U}$ the stable locus in $T^*\mathfrak G$, once we restrict to $\text{U}$, we may apply the reduction procedure given by Proposition 2.8 of \cite{KR} to obtain a $\mathscr W$-algebra, denoted $\mathscr A_c$, on $\mathcal M^s/G \cong \text{Hilb}^n(\mathbb C^2)$. Let $\tilde{\mathscr A}_c$ denote the $\mathscr W$-algebra obtained from $\mathscr A_c$ by extending the coefficients $\mathbb C((\hbar))$ by a square root of $\hbar$. Then there is a natural functor from $e\mathcal H_c e$-modules to $\tAc$-modules with an $F$-action given by:
\[
\text{Loc}_{\tAc}(M) = \tAc \otimes_{e\mathcal H_c e} M.
\]
This uses the deformed Harish-Chandra homomorphism to construct a natural action of $e\mathcal H_c e$ on $\tAc$ (\textit{i.e.} the action is induced by the $e\mathcal H_c e$-action on the $\mathcal D_{\mathfrak G}$-module $\mathcal L_c$).

By \cite[Proposition 2.8]{KR}, the construction of the $\mathscr W$-algebra on $\text{Hilb}^n(\mathbb C^2)$ has the property that its modules are equivalent to $(G,c.\text{tr})$-equivariant $\mathscr W_{U}$-modules with an $F$-action that are supported on $\mathcal M^s$. Thus comparing the definitions, we see that given a module in category $\mathcal O_c$, the $\mathscr W_{\text{U}}$-module we obtain by pulling back $\text{Loc}_{\tAc}(M)$ to $\text{U}$ is exactly $\mathscr W_{\text{U}}\otimes_{\E_{\text{U}}}\mu_{\text{U}}(\text{Loc}(M))$. Now both $\text{Loc}$ and $\text{Loc}_{\tAc}$ have left adjoints $\mathbb H$ and $\mathbb H_{\tAc}$ given by taking twisted invariant global sections, and moreover, since by \cite[Theorem 1.5]{GG} and \cite[Lemma 4.7]{KR} there are filtration preserving isomorphisms $\text{End}(\mathcal L_c)^{\text{op}} \cong e\mathcal H_c e \cong \text{End}_{\text{Mod}^{\text{good}}_F(\tAc)}(\tAc)^{\text{op}}$ it follows we must have $\mathbb H_{\tAc}(\mathscr W_{\text{U}}\otimes_{\E_{\text{U}}}\text{Loc}(M)) = \mathbb H(\text{Loc}(M))$.

\noindent
Now let 
\[
\mathscr Y = \{a/b: 2\leq b \leq n, a<0,  \text{ g.c.d.}(a,b)=1\}.
\]
Then Theorem $4.9$ of \cite{KR} shows that if $c \notin \mathscr Y$ the functor $\mathbb H_{\tAc}$ gives an equivalence between the category of coherent good $\tilde{\mathscr A}_c$-modules with an $F$-structure and the category of finitely generated $e\mathcal H_c e$-modules, with quasi-inverse $\text{Loc}_{\tAc}$. But now if $\mathcal F$ is an object in $\mathcal C_c$, then the natural map $\mathcal F \to \text{Loc}(\mathbb H(\mathcal F))$ must be an isomorphism over $\text{U}$ (since $\mathscr W_X$ is faithfully flat over $\E_X$). Note that this implies that the functors $\mathcal{KZ}_\lambda$ which we have defined vanish on $\text{ker}(\mathbb H)$, and so descend to the quotient category $\mathcal C_c /\text{ker}(\mathbb H)$.

\subsection{}
The case $\lambda = (1^n)$ corresponds to the original $KZ$-functor of \cite{GGOR}, as we now show. Let $\mathfrak G^{\text{rss}}$ be the set of pairs $(X,v)\in \mathfrak G$ such that $X$ is regular semisimple and $v$ is a cyclic vector for $X$. This is a smooth locally closed subvariety of $\mathfrak G^\circ$, and by the characteristic cycle bound, if $\mathcal F \in \mathcal C_c$ then, viewing $\mathcal F$ as before as a $(\text{GL}(V),c.\text{tr})$-equivariant $\mathcal D_{\mathfrak G^\circ}$-module, its restriction $\mathcal F_{|\mathfrak G^{\text{rss}}}$ must be a local system. As in Lemma \ref{microlocalpi1} we see that that $\pi_1(\mathfrak G^{\text{rss}})$ is an extension of $\mathcal B_n$ by $\mathbb Z \cong \pi_1(\text{GL}(V)$, where $B_n$ is the braid group on $n$ strands, so that the functor $KZ_{(1^n)}$ given by descending the local system $\mathcal F_{|\mathfrak G^{\text{rss}}}$ to $\tilde{Z}_{(1^n)}$.

 \begin{prop}
 \label{KZcheck}
 The functors $KZ$ and $KZ_{(1^n)}$ are naturally isomorphic. 
 \end{prop}
\begin{proof}
This follows immediately from the construction of the functors. Consider the diagram:

\xymatrix{ & & & &  \mathfrak G^{\text{reg}} \ar[r]^\pi & \mathfrak h/W &  \ar[l]_p \mathfrak h \\
& & & &  \mathfrak G^{\text{rss}} \ar[u]^j \ar[r]^{\pi_\text{r}} & \hreg/W  \ar[u] &  \ar[l]_{p_\text{r}} \hreg \ar[u]_i}

\noindent
The construction of the $KZ$-functor uses the right-hand square in the diagram. Given a module $M$ in category $\mathcal O_0$, we may localize to $\hreg$, and via the Dunkl homomorphism view it as a $W$-equivariant $\mathcal D$-module on $\hreg$ (which we will again write as $M$). As such there is a $\mathcal D$-module $N$ which is a vector bundle with connection on $\hreg$ such that $M_{|\hreg} \cong p^*(N)$. 

The global sections of $N$ become isomorphic to $eM_{|\hreg}$ as an $\mathcal D_{\hreg}^W$-module after conjugation by $\delta$: over $\hreg$ the map to $\mathfrak h/W$ is a Galois covering with group $W$, so locally $M$ is just $|W|$ copies of $\mathcal O_\mathfrak h$, and the sections of $N$ then correspond to $eM$, but the action of differential operators need to be conjugated by the Jacobian of the quotient map, which in this case is $\delta$ (c.f. \cite{HS} for a discussion of this in the context of push-forward of $\mathcal D$-modules -- this is the reason we must twist by the action of $\delta$ in the definition of $\Psi_c$).

We now consider the other side of the diagram. The construction of the $\mathcal D$-module, $\mathcal M$ say, in $\mathcal C_c$ from $M$ proceeds via the isomorphism $\Psi_c$ from $e\mathcal H_c e$ to $(\mathcal D_c(\mathfrak X)/\mathcal D_c(\mathfrak X)\Lieg_c)^{\text{SL}(V)}$, so that
\[
\mathcal M = \mathcal D_{\mathfrak X,c}\otimes_{\Psi_c} eM.
\]
On $\mathfrak G^{\text{rss}}$ the map $\Psi_c$ discussed in section \ref{deformed} gives an isomorphism between this algebra and $\mathcal D(\hreg)^W$, where the conjugation by $s^c$ used in the deformed Harish-Chandra homomorphism corresponds the tensor product by $L_{q^{-1}}$ in our construction of $KZ_{(1^n)}$. Then the above discussion shows that we have $j^*(\mathcal M) \cong \pi_\text{r}^*(N)$, and hence the monodromy measured by $KZ_{(1^n)}$ is exactly that measured by $KZ$. \end{proof}

\begin{remark}
A fundamental conjecture due to Kashiwara (a proof of which has recently been announced by Kashiwara and Vilonen \cite{KV}) states that a regular holonomic $\mathcal E_X$-module extends uniquely beyond an analytic subset of codimension greater or equal to three, which means that the structure of the stack of such $\mathcal E_X$-modules is captured in codimension two. Loosely,  the ``codimension zero'' information is given by a collection of local systems, and the codimension $1$ information provides certain ``glue'' between them, with codimension two imposing additional constraints on objects. This implies that one can understand regular holonomic $\mathcal D_X$- modules whose singular support lies in a certain Lagrangian $\Lambda$ via the geometry of $\Lambda$ up to codimension two singularities, and in the topological context this idea has been studied in \cite{GMV}. One motivation for this paper came from a desire to get such a ``microlocal'' understanding of the category $\mathcal C_c$ (note that it is shown in \cite[\S 5.3]{GG} that the $\mathcal D$-modules in the category $\mathcal C_c$ are regular holonomic -- indeed this follows by a standard argument from the finiteness of the number of orbits of $\text{PGL}(V)$ on $\mathcal N \times \PP(V)$). Note that the topological case and its relation to $\mathcal E_X$-modules has also been studied in \cite{W1}, \cite{W2}.
\end{remark}

\section{Characteristic cycles of standard modules.}
\label{CCcomputation}
\subsection{}
Having defined microlocal KZ functors, we now attempt a first study of them. For this, as for the original KZ functor, we use the standard modules, whose construction we now recall. These modules are analogous to Verma modules in Lie theory: they are defined for any value of $c$, and are generically simple. 

Note that $\mathbb C[\mathfrak h^*]\rtimes \mathbb C[W]$ is a subalgebra of $\mathcal H_c$, and it has a natural evaluation homomorphism $\text{ev} \colon \mathbb C[\mathfrak h^*]\rtimes \mathbb C[W] \to \mathbb C[W]$ given by $P\otimes w \mapsto P(0).w$. Thus pull-back via $\text{ev}$ allows us to lift modules for $\mathbb C[W]$ to modules for $\mathbb C[\mathfrak h^*]\rtimes \mathbb C[W]$. Given $\tau$ a representation of $W$ we define a standard module\footnote{This construction of course works for any complex reflection group, not only symmetric groups.} $M(\tau)$ by inducing the pull-back of $\tau$ to $\mathcal H_c$:
\[
M(\tau) = \mathcal H_C \otimes_{\mathbb C[\mathfrak h^*]\rtimes \mathbb C[W]} \text{ev}^*(\tau).
\]
If $\tau$ is an irreducible representation indexed by a partition $\lambda$, then we write $M_\lambda$ for $M(\tau)$. 

The standard modules have a unique simple quotient $L_\lambda$, and by the general machinery of highest weight categories one can also establish a Brauer-Bernstein-Gelfand type reciprocity formula relating the multiplicity of simple modules in standards to the multiplicity of standard modules in projective modules, see \cite{Gu}. It follows easily from this that $\mathcal O_c$ is semisimple if and only if the standard modules are simple. We wish to calculate the dimensions of the local systems given by the functors $\text{KZ}_\mu$ on these modules, or in other words (see $\S$\ref{mKZdef}) we wish to calculate the multiplicities of the characteristic cycle of these modules on the corresponding component of $\Lambda$. 

\subsection{}
We now wish to show that the characteristic cycle of a standard module is independent of the parameter $c$. To do this we use the ``Harish-Chandra'' module $F_{HC}$ defined by Gan and Ginzburg \cite[\S 7.4]{GG}. This is the $\mathcal D_c(\mathfrak X)$-module defined by 
\[
F_{HC} = \mathcal D(\mathfrak X,c)/(\mathcal D(\mathfrak X,c)\mathfrak g_c + \mathcal D(\mathfrak X,c)\mathfrak Z_+),
\]
where $\mathfrak Z_+$ denote augmentation ideal of $\mathcal Z$ the algebra of $\text{ad}(\Lieg)$-invariant constant coefficient differential operators on $\mathfrak g$.  We first show that the characteristic cycle of $F_{HC}$ is independent of $c$.

\begin{lemma}
The characteristic cycle of $F_{HC}$ is given by the ideal $I = \langle \sigma(x), \sigma(z): x \in \Lieg, z \in \mathcal Z_+\rangle$, where $\sigma$ denotes the symbol map, that is, the generators $\Lieg$ and $\mathcal Z_+$ are involutive. Thus $\text{CC}(F_{HC})$ is independent of $c$. 
\end{lemma}
\begin{proof}
To see that our generators are involutive, we use the following criterion \cite[\S 2.2]{Kbook}: if $\{x_i: 1 \leq i \leq m\}$ are elements of $\mathcal D$ of order $m_i$ which satisfy the conditions that 
\begin{enumerate}
\item $\bigcap_{i=1}^m \sigma(x_i)^{-1}(0)$ is of codimension $m$;
\item $[x_i,x_j] = \sum_{k} q_{ijk}x_k$, for some $q_{ijk} \in F_{m_i+m_j-m_k-1}$;
\end{enumerate}
then the system of generators is involutive. 

In our case, take $\{x_i\}$ to be a basis of $\Lieg$ (so of order $1$) and the standard generators of $\mathcal Z_+$ (of order $2,\ldots, n-1$). The second condition is then clearly satisfied, since the operators in $\mathfrak Z_+$ commute with each other and $\Lieg_c$. For the first condition, note that the associated graded generators define the set
\[
\{(X,Y,v,w) \in \Lieg \times \Lieg \times V \times V^*: \mu(X,Y,v,w) = 0, y \text{ nilpotent}\},
\]
where $\mu$ is the moment map, that is $\mathcal M_{\text{nil}}$ (in fact with a nonreduced scheme structure), so the first condition is satisfied (thus the lemma reduces essentially to the flatness of the moment map). The independence of $c$ follows immediately.
\end{proof}

\noindent
Recall the definition of $\mathscr Y$ in $\S$\ref{microsupportH}. In the following Lemma we use the fact that we may calculate the characteristic cycle of a $\mathcal D$-module either by taking a good filtration and calculating the multiplicity of the resulting $\mathcal O$-coherent sheaf on the components of its support, or by taking the multiplicities on the components of the support of the associated $\mathcal E$-module. See \textit{e.g.} \cite[Chapter 7]{Kbook} for a discussion of this fact at the level of supports.

\begin{lemma}
\label{indofc}
Suppose that $c \notin \mathscr Y$.
Let $M_\lambda$ be a standard module, and $\mu$ a partition of $n$. Then $\dim(KZ_\mu(M_\lambda))$ is independent of $c$. 
\end{lemma}
\begin{proof}
Let $\text{co}(\mathfrak h^*)$ be the coinvariant module $\mathbb C[\mathfrak h^*]/\langle I \rangle$, where $I$ is the augmentation ideal in $\text{Sym}(\mathfrak h)^W$, viewed as a representation of $\mathbb C[\mathfrak h^*]\rtimes W$ in the obvious way, and let $C$ denote the $\mathcal H_c$ module induced from $\text{co}(\mathfrak h)$. It is shown in \cite[\S 7]{GG} that if $\langle I \rangle$ denotes the ideal of $e\mathcal H_c e$ generated by $\text{Sym}(\mathfrak h)^W_+$, then $eC \cong e\mathcal H_c e/\langle I \rangle$.  Using this, they show that $\mathbb H(F_{HC}) = eC$. Now the discussion in $\S$\ref{microsupportH} shows that if $c \notin \mathscr Y$ then the intersection of the characteristic cycle of $F_{HC}$ with $\text{U}$ is completely determined by $\mathbb H(F_{HC})$, and thus it coincides with that of $\text{Loc}(eC)$, (as $\mathbb H \circ \text{Loc}$ is naturally isomorphic to the identity). It follows that the characteristic cycle of $\text{Loc}(eC)$ intersected with $\text{U}$ is independent of $c$.

Now consider $R = \Delta(\mathbb C[W])$, the representation of $\mathcal H_c$ induced from the regular representation of $W$. Since it is clear that we may filter $\text{co}(\mathfrak h^*)$ in such a way that the associated graded module is isomorphic to $\mathbb C[W]$ as a $\mathbb C[\mathfrak h^*]\rtimes W$-module, and $e\mathcal H_c$ is a flat right $\mathbb C[\mathfrak h^*]\rtimes W$-module, it follows that $R$ and $C$ have the same class in the Grothendieck semigroup for $\mathcal O_c$, and hence the same characteristic cycle.

Since $R$ is clearly finitely generated as a $\mathbb C[\mathfrak h]$-module, and $\mathbb C[\mathfrak h]$ is finite over $\mathbb C[\mathfrak h]^W$, the module $\text{Loc}(eR) = \mathcal L_c\otimes_{e\mathcal H_c e} eR$ is finite over $\mathbb C[\mathfrak X]^G$. It follows that if $\Gamma_i$ is a good filtration on $\mathcal L_c$, the filtration $\Gamma_i\otimes eR$ is good for $\text{Loc}(eR)$. Thus we may use the order filtration on $\mathcal D_c(\mathfrak X)$ to obtain a good filtration on $\mathcal L_c$.

Now $\mathbb C[W]$ is of course a bimodule for $W$, hence it carries a right action of $W$ which commutes with the left action used in the construction of $R$. Thus we have a right $W$ action on $eR$ via $e\mathcal H_c e$-module automorphisms, and so the localization $\text{Loc}(eR)$ carries a $W$-action as a $\mathcal D_{\mathfrak X,c}$-module, which is compatible with our filtration.  But now taking the isotypic components according to this action yields the standard modules, and hence their characteristic cycles are determined by this action on $\text{CC}(\text{Loc}(eR))=\text{CC}(\text{Loc}(eC))$, which must therefore be independent of $c$ as required.
\end{proof}

\subsection{}
The previous lemma shows that, in order to calculate the part of the characteristic cycles of $\text{Loc}(eM_\lambda)$ lying in $\Lambda^s$ we may assume that $c=0$. It is known \cite{BEG} that the standard modules are simple whenever the finite Hecke algebra at $q=e^{2\pi i c}$ is semisimple. Thus it follows the standard modules are simple when $c=0$. We will write $\mathfrak M_\lambda$ for the $\mathcal D$-module corresponding to the standard module $eM_\lambda$ when $c=0$. We now wish to identify $\mathfrak M_\lambda$ more explicitly.

Recall the Grothendieck simultaneous resolution $\mu\colon \tilde{\Lieg} \to \Lieg$ where 
\[
\tilde{\Lieg} = \{(x,F) \in \Lieg \times \mathcal B: F = (F_i)_{1 \leq i \leq n}, x(F_i) \subseteq F_i\}
\]
\noindent
and $\mu$ is the projection to the first factor. Over the subset $\Lieg^{rss}$ of regular semisimple elements of $\Lieg$ this is a $W$-covering. Moreover, it is know by \cite[\S 4]{HK} that the push-forward $F_\Lieg = \mu_*(\mathcal O_{\tilde{\Lieg}})$ is a semisimple holonomic $\mathcal D$-module which is the minimal extension of its restriction to $\Lieg^{rss}$, where it is just the local system corresponding to the principal $W$-covering. For $\lambda$ an irreducible representation of $W$, let $\mathfrak M_\lambda^{\Lieg}$ denote the simple $\mathcal D_{\Lieg}$-module corresponding to the minimal extension of the local system given by $\lambda$ on $\Lieg^{rss}$. Since we are assuming that $c=0$, the sheaf of rings $\mathcal D_{\mathfrak X,0}$ is just $\mathcal D_{\mathfrak X} = \mathcal D_{\Lieg}\boxtimes\mathcal D_{\mathbb P(V)}$, and we may pull-back $F_\Lieg$ to a $\mathcal D_{\mathfrak X}$-module $F = F_\Lieg \boxtimes \mathcal O_{\mathbb P(V)}$. It follows that $F$ is a direct sum of simple modules $\mathfrak M^\Lieg_\lambda\boxtimes \mathcal O_{\mathbb P(V)}$.

\begin{lemma}
The sheaf $\mathfrak M_\lambda$ arising from the standard module $M_\lambda$ at $c=0$ is isomorphic to $\mathfrak M_\lambda^\Lieg\boxtimes \mathcal O_\lambda$.
\end{lemma}
\begin{proof}
We use the (original) $KZ$-functor. It is known by \cite[\S 6.2]{GGOR} that the $KZ$-functor applied to a standard modules $M_\lambda$ yields the corresponding cell module $S_\lambda$ for the finite Hecke algebra. Since we are assuming $c=0$ this in turn is just the corresponding irreducible representation of $W$. Thus using Proposition \ref{KZcheck} it follows that $\mathfrak M_{\lambda|\mathfrak X^{rss}}$ is just the local system given by $\lambda$. Since $\mathfrak M_\lambda$ is simple, it follows that the sheaf $\mathfrak M_\lambda$ is just the minimal (or Goresky-MacPherson) extension of this local system to $\mathfrak X$. However, the sheaf $\mathfrak M_\lambda^{\Lieg}\boxtimes \mathcal O_{\mathbb P(V)}$ clearly restricts to give the same local system and is also simple, hence we see that $\mathfrak M_\lambda = \mathfrak M_\lambda^\Lieg$ as required.
\end{proof}

\subsection{}
To compute the characteristic cycles of the standard modules, we will actually use a Fourier dual description of the variety $\Lambda$. Thus we need to use the following standard remark: let $X = E \times B$ be a trivial vector bundle over $B$ and let $\mathcal D^\text{mon}_X$ be the category of monodromic $\mathcal D$-modules on $X$, that is, those for which the action of the Euler vector field associated to $E$ is locally finite. Then it is known that the partial  Fourier transform $\mathcal F_E$ along $E$ gives an equivalence between $\mathcal D^{\text{mon}}_{E\times B}$ and $\mathcal D^{\text{mon}}_{E^*\times B}$. 

\begin{lemma}
We may canonically identify 
\[
T^*(E\times B) \cong T^*(E^*\times B)\cong E\times E^*\times T^*B. 
\]
Under this isomorphism, if $N$ is a monodromic $\mathcal D_X$-module we have 
\[
\text{CC}(N) \cong \text{CC}(\mathcal F_E(N))
\]
\end{lemma}
\begin{proof}
For a discussion of this see for example \cite[\S 3]{HK}. 
\end{proof}

We apply this to $\mathfrak X = \Lieg \times \mathbb P(V)$ with $E= \Lieg$ and $B = \mathbb P(V)$. The isomorphism $T^*(\Lieg\times \mathbb P(V)) \cong T^*(\Lieg^* \times \mathbb P(V))$ is simply given by identifying both with $\Lieg \times \Lieg^* \times T^*\mathbb P(V)$ (or via the trace form $\Lieg \times \Lieg \times T^*\mathbb P(V)$). We want to understand the components of $\Lambda$ in this ``Fourier dual'' picture. This is essentially already done in \cite{GG} but note they use a different transform -- the bundle maps we use are given by
\[
\begin{split}
p\colon T^*(\Lieg\times \mathbb P(V)) \to \Lieg \times \mathbb P(V), \quad & (X,Y,i,j) \mapsto (X,i) \\
\check{p} \colon T^*(\Lieg^* \times \mathbb P(V)) \to \Lieg^*\times \mathbb P(V),\quad & (X,Y,i,j) \mapsto (Y,i). 
\end{split}
\]
whereas in \cite{GG} they use $\check{p}(X,Y,i,j) = (Y,j)$. Nevertheless the same argument they use gives the following:

\begin{lemma}
Let $\Lambda^*$ denote the image of $\Lambda$ under the canonical isomorphism in $T^*(\Lieg^*\times \mathbb P(V))$. The $\Lambda^*$ is the union of the conormal bundle to the $\text{PGL}(V)$-orbits on $\mathcal N \times \mathbb P(V)$. 
\end{lemma}
\begin{proof}
This follows from the fact that any conic Lagrangian in a cotangent bundle $T^*X$ is the closure of the conormal bundle of a smooth locally closed subvariety of $X$, and the fact that $\text{PGL}(V)$ acts with finitely many orbits on $\mathcal N \times \mathbb P(V)$.
\end{proof}

Thus the $\mathcal D$-modules on $\Lieg^*\times \mathbb P(V)$ which we study are supported on $\mathcal N \times \mathbb P(V)$ and are smooth along $\text{PGL}(V)$-orbits. Note that since there are finitely many orbits, they give a Whitney stratification of $\mathcal N \times \mathbb P(V)$. For our purposes we need to check which $\text{PGL}(V)$-orbits correspond to the components in $\Lambda^s$.

\begin{lemma}
\label{Fourierdualofcomponents}
The component $\Lambda^s_\lambda$ ($\lambda \vdash n$) of $\Lambda^s$ corresponds to the $\text{PGL}(V)$-orbits 
\[
\mathcal O_{\lambda}^{\text{cyc}} = \{(Y,\ell): Y \text{ has Jordan type } \lambda, \text{ and } \Lieg_Y.\ell = V\},
\]
where $\Lieg_Y$ denotes the centralizer of $Y$ in $\Lieg$.
\end{lemma}
\begin{proof}
Recall from the last paragraph of Section \ref{decompositionofM} that $\Lambda^s_\lambda$ corresponds to the component $\Lambda_{\lambda,\emptyset}$ of $\Lambda$, that is the closure of the conormal bundle $T_S^*(\mathfrak X)$ where $S$ is the stratum of pairs $(X,\ell)$ where $X$ has eigenvalues with multiplicity given by $\lambda$ and $\ell$ is cyclic for $X$. Then the proof of Lemma \ref{commuting} shows that the set of nilpotent endomorphisms $Y$ which commute with $X$, have Jordan type $\lambda$, and for which the projection of $\ell$ to each generalized eigenspace of $X$ is cyclic for the restriction of $Y$ to that space, is open dense in $\check{p}(\Lambda_{\lambda,\emptyset})$. Since $\Lieg_Y$ contains $\mathbb C[X,Y]$ it follows $\Lieg_Y.\ell = V$ for such $Y$, and these clearly form a single $\text{PGL}(V)$ orbit. 
\end{proof}

Note that if $\mathcal O_\lambda$ denotes the orbit of nilpotent endomorphisms of Jordan type $\lambda$, then clearly
\[
\mathcal O_{\lambda}^{\text{cyc}} = \Lambda^s \cap \mathcal O_{\lambda}\times V.
\]
where we have identified $\Lambda^s$ with a subset of $\Lambda^*$ using the isomorphism above.
\subsection{}
Since the characteristic cycle of $\mathfrak M_\lambda$ is determined by its deRham complex $\text{DR}(\mathfrak M_\lambda)$ we may work purely topologically. 

\begin{Theorem}
Suppose that $c \notin \mathscr Y$. Let $\mathfrak M_{\lambda,c}$ be the $\mathcal D_{\mathfrak X,c}$-module associated to the standard module $M_\lambda$. Then we have
\[
\text{CC}(\mathfrak M_{\lambda|\mathfrak X^{\text{reg}}}) = \sum_{\mu \leq \lambda} K_{\lambda,\mu} [\Lambda_\mu],
\]
where $\leq$ denotes the dominance order, and $K_{\lambda,\mu}$ are the Kostka numbers. In other words, we have
\[
\dim(KZ_\mu(eM_\lambda) = K_{\lambda,\mu}.
\]
\end{Theorem}
\begin{proof}
We have already seen that it is enough to compute the characteristic cycle in the case $c=0$, where we write $\mathfrak M_\lambda$ instead of $\mathfrak M_{\lambda,0}$.
Taking the Fourier transform $\mathcal F_\Lieg(\mathfrak M_\lambda)$ we obtain a $\mathcal D$-module on $\mathcal N \times \mathbb P(V)$. From the preceeding discussion we know that $\mathfrak M_\lambda = \mathfrak M^\Lieg_\lambda \boxtimes \mathcal O_{\mathbb P(V)}$, and its Fourier transform is then just $\mathcal F_{\Lieg}(\mathfrak M^{\Lieg}_\lambda)\boxtimes \mathcal O_{\mathbb P(V)}$. Now by \cite[\S 5]{HK} the $\mathcal D$-module $\mathcal F_\Lieg(\mathfrak M^{\mathfrak g}_\lambda)$ corresponds to the intersection cohomology sheaf $\mathcal I_\lambda$ on $\mathcal O_\lambda$, so that we are reduced to calculating the characteristic cycles of intersection cohomology sheaves of nilpotent orbits. But these are known thanks to \cite{EM} which uses a torus symmetry trick and the knowledge of the stalk Euler characteristics of the complexes. Indeed \cite[Lemma 3.2]{EM} shows that for nilpotent orbits $\alpha,\beta \subset \mathcal N$ (in type $A$) the local Euler obstruction $c_{\alpha,\beta}$ is zero whenever $\alpha \neq \beta$ (and by definition it is $1$ when $\alpha=\beta$). Thus since our intersection cohomology sheaves $\mathcal I_\lambda$ are constructible with respect to the stratification of $\mathcal N$ it follows from the Index Theorem (see Theorem \ref{indextheorem}) that 
\[
CC(\mathcal I_{\lambda}) = \sum_{\mu} \chi_\mu(\mathcal I_\lambda)[T^*_{\mathcal O_\mu}\mathfrak g],
\]
where $\chi_\mu(\mathcal I_\lambda)$ denotes the stalk Euler characteristic of $\mathcal I_\lambda$ along the orbit $\mathcal O_\mu$. But it is shown in \cite[\S 2]{L81} that the dimensions of the stalk cohomologies of the intersection cohomology sheaves of nilpotent orbits in $\GL(V)$ are given by Kostka polynomials, from which it follows immediately from this that the stalk Euler characteristics are given by the Kostka numbers $K_{\lambda, \mu}$, which completes the proof.
\end{proof}

\begin{remark}
Note that the real advantage of the description of $\text{Loc}(eM_\lambda)_{|c=0}$ as $\mathfrak M_\lambda^\Lieg \boxtimes \mathcal O_{\mathbb P(V)}$ was that it makes the computation of the Fourier transform straightforward. For details on the index theorem  which reduce the computation of characteristic cycles to that of Euler characteristics and the local Euler obstructions see the Appendix.
\end{remark}

\begin{remark}
In \cite[\S 2.7]{GS} Gordon and Stafford defined a characteristic cycle for modules in category $\mathcal O^\text{sph}_c$ via their $\mathbb Z$-algebra construction, and computed them for standard modules. Subsequently, Ginzburg, Gordon and Stafford \cite{GGS} established the equivalence of their definition with the $\mathcal D$-module one given in \cite{GG}, which implies the above result. Their calculation however depended on Haiman's work on the Hilbert scheme, while the above calculation uses only standard tools from $\mathcal D$-modules and the elegant torus action trick of Evens-Mirkovic.
\end{remark}

\section{Factorization through Hecke algebras}
\label{HeckeAlgebras}
\subsection{}
\label{nearbyvanishing}
We now wish to use $\mathcal D$-module techniques to show that the $KZ$-functor factors through a functor to representations of the finite Hecke algebra. Thus we wish to check that the generators of the braid group satisfy the quadratic relation 
\[
(T_s -1 )(T_s + q) = 0.
\]
Since this relation involves only a single generator, we first establish it in the rank one case. For this we need to recall the classification of regular singularities $\mathcal D$-modules on the complex line which are $\mathcal O$-coherent on $\mathbb C-\{0\}$. Let $\mathcal M$ be such a $\mathcal D$-module. Then by taking a $V$-filtration of $V^\bullet\mathcal M$ (see for example \cite{K83a}) one may define functors of nearby and vanishing cycles:
\[
\Psi(\mathcal M) = \text{gr}^V_0(\mathcal M), \quad \Phi(\mathcal M) = \text{gr}^V_{-1}(\mathcal M),
\]
from $\mathcal D$-modules to finite-dimensional vector spaces. Moreover these functors are equipped with a ``canonical''' map $c\colon \Psi(\mathcal M) \to \Phi(\mathcal M)$ (induced by the action of $\partial$), and a ``variation'' map $v\colon \Phi(\mathcal M) \to \Psi(\mathcal M)$. These are related to each other by the equations:
\begin{equation}
\label{canvarequations}
c\circ v = T - \text{Id}, \quad v \circ c = T - \text{Id}.
\end{equation}
where $T$ is the monodromy operator $\text{exp}(-2i\pi x\partial)$.
In fact it can be shown that our category of $\mathcal D$-modules is then equivalent to the category of pairs of vector spaces $(V,W)$ together with linear maps $c\colon V \to W$, and $v\colon W \to V$ such that $1$ is not an eigenvalue of either $cv$ or $vc$. 

\begin{lemma}
Suppose that $n=2$. The monodromy representations of the braid group $\mathcal B_2 = \mathbb Z$ given by $KZ_{(1^2)}$ factor through the Hecke algebra at $q=e^{2\pi i c}$.
\end{lemma}
\begin{proof}
By taking co-ordinate ``centre of mass'' coordinates $x,x',y,y'$ such that $W=S_2$ acts trivially on $x',y'$ and by the sign representation on $x,y$, we may reduce to the case of $T^*\mathbb P^1$. By the discussion in the last section of \cite{KR}, the sheaf of rings on $\text{Hilb}^2(T^*\Sigma)$ localizing $e\mathcal H_c e$ when restricted to the exceptional divisor is just the twisted ring of differential operators $\mathcal D_{\mathbb P^1,\lambda}$ where $\lambda = c - \frac{1}{2}$. The characteristic cycle of a module $\mathcal M$ in $\mathcal C_c$ must then lie in $\mathbb P^1 \cup T^*_{0}\mathbb P^1$, and $\mathcal M$ becomes a module for the sheaf of rings $\mathcal D_{\mathbb P^1,\lambda}$. We claim that this implies the monodromy around $T^*_0\mathbb P^1 - \{0\}$, which is the monodromy given by $KZ_{(1^2)}$, satisfies the quadratic relation of the Hecke algebra.

In terms of the sheaf of rings $\mathcal D_{\mathbb P^1,\lambda}$, this monodromy is the action of $T$ on the vanishing cycles sheaf. Since the monodromy operator is given by $\text{exp}(2\pi i x \partial_x)$, the twisted equivariance forces $T$ to act as  $e^{2\pi i \lambda} = -e^{2\pi i c} =-q$ on the nearby cycles sheaf. But now it follows that $v\circ c =-q-1$. By Equation \ref{canvarequations} we see that this implies that on the vanishing cycles we have
\[
(T-\text{Id})^2 = (c\circ v)^2 = c\circ(v\circ c) \circ v = c\circ( -q-1)\circ v = -(q+1)(T-\text{Id}).
\]
Rearranging this equation gives us the Hecke algebra quadratic equation as required. 
\end{proof}

\subsection{}
We now deduce the general case by reduction to the rank one case. Let $\mathfrak g_i$ denote the subset of $\Lieg$ consisting of matrices of the form
\[
\left(\begin{array}{ccccc} a_1 & 0 &  \cdots & \cdots & 0 \\0 & \ddots & 0 & \cdots & 0 \\ \vdots & 0 & A & 0 & \vdots \\ \vdots & \vdots & 0 & a_{i+2} & 0 \\ 0 &\cdots & \cdots & 0 & \ddots \end{array}\right)
\]
where the $a_j$ ($j \in \{1,\ldots,i-1,i+2,\ldots n\}$) are pairwise distinct numbers, $A$ is a $2\times 2$ matrix with entries in the $i$th and $(i+1)$th rows, and the eigenvalues of $A$ are distinct from the $a_j$s. (The set $\mathfrak g_1$ is also used to reduce to rank $1$ in \cite[\S3]{KR}.) It is easy to check that $\mathfrak g_i\times V$ is noncharacteristic for $\Lambda$, thus if $M$ is an object in $\mathcal C_c$, its pull-back to $\mathfrak g_i$ is holonomic with characteristic variety contained in the fibrewise projection of $\Lambda$ to $T^*(\mathfrak g_i\times V)$. More precisely, letting $i \colon \mathfrak g_i \times V^\circ \to \mathfrak G^\circ$ denote the inclusion map we have the diagram:

\xymatrix{ & & &  T^*(\mathfrak g_i \times V^\circ) & \ar[l]_{i_d \qquad} \mathfrak X \times_{\mathfrak g_i\times V^\circ} T^*(\mathfrak g_i \times V^\circ) \ar[d]^{i_\pi}  \\
& & & &  T^*\mathfrak G^\circ}
\noindent
and $\text{Ch}(i^*(M)) \subset i_di_{\pi}^{-1}(\text{Ch}(M)$. Here $i_\pi$ is just the obvious inclusion map, and $i_d$ is the fibrewise projection map. Now it is immediate from the definitions that $i_di_\pi^{-1}(\Lambda)$ can be identified with $\Lambda_2 \times \mathbb C^{n-2}\times (\mathbb C^\times)^{n-2}$, where $\Lambda_2$ is the variety $\Lambda$ in the case $n=2$.  Moreover this identification is equivariant for the action of the group $G_2 = T\times \text{SL}_2$ where $G_2$ is embedded in $\text{GL}(V)$ in the obvious way so as to be compatible with the map $i$. Let
\[
A(t) = \left(\begin{array}{cc} e^{2\pi i t} & 0 \\ 0 & e^{-2\pi i t}\end{array}\right)
\]
Picking pairwise distinct complex numbers $(a_j)$ none of which lie in $\{\pm 1\}$, the action of the generator $T_i$ on $KZ_{(1^n)}(M)$ is given by monodromy around the curve
$t \mapsto i(a_1,\ldots,a_{i-1},A(t),a_{i+2},\ldots, a_n)$. Thus we see that it is enough to check the quadratic relation holds for the case $n=2$, which has already been checked. Thus we have the following theorem

\begin{Theorem}
Let $M \in \mathcal C_c$. The monodromy representation given by $KZ_{(1^n)}(M)$ factors through the Hecke algebra $H_W(q)$ where $q = e^{2\pi i c}$.
\end{Theorem}

This recovers, in type $A$, Theorem $5.13$ of \cite{GGOR} which is one of the central results of that paper. The proof given there is quite different -- it checks the result directly on standard modules, and then deduces the general case by a deformation argument.

\begin{remark}
The result of this section applies only to the original $KZ$-functor. At the moment the author does not have an analogue of this result for the microlocal $KZ$-functors. One reason for this is that it is harder to give a presentation of the groups $\mathcal B_\lambda$. However since $\mathcal C_c$ is Artinian it follows formally that the action of $\mathcal B_\lambda$ will factor through a finite-dimensional algebra. Note that Rouquier \cite{R} has shown that the category $\mathcal O_c$ is equivalent to the category of representations of the $q$-Schur algebra. Since the $q$-Schur algebra is a quotient of the quantum group $U_v(\mathfrak{gl}_n)$, its representations carry Lusztig's braid group action, and the partitions $\lambda$ can be interpreted as weights of such representations. The subgroup of Lusztig's braid group acting on the weight space corresponds to $\mathcal B_\lambda$ and we conjecture that given $M$ a representation in $\mathcal O_c$, the representation $KZ_\lambda$ is given by the $\lambda$-weight space of the corresponding representation of the $q$-Schur algebra equipped with this braid group action.
\end{remark}

\section{On the minimal resolution of $\mathbb C^2/\mu_l$ as a hypertoric variety}
\label{cyclic}

\subsection{}
The ideas of this paper should have applications to a larger class of rational Cherednik algebras than that which is attached to the symmetric groups. To illustrate this, in this section we consider the case of the rational Cherednik algebra $\mathcal H_\kappa$ attached to a cyclic group $\mu_l$ of order $l$, which we regard as a complex reflection group (with the obvious one-dimensional representation). We use quantum  Hamiltonian reduction on a moduli space of representations for the cyclic quiver as in \cite{Ku} and \cite{BK}, which give a deformation of the minimal resolution of a Kleinian singularity. Each components of the exceptional divisor of the resolution yields an exact functors on category $\mathcal O$ for $\mathcal H_\kappa$, and moreover, taken together, they can be used to construct a functor from $\mathcal O_c$ to the category of representations of the ``cyclotomic $q$-Schur algebra'' (\textit{i.e.} the quasi-hereditary cover of \cite[\S 6]{DJM}) of the corresponding cyclotomic Hecke algebra. This gives a geometric perspective on the the ideas of \cite{GGOR} and \cite{R}. Indeed it is shown in \cite[Theorem 5.16]{GGOR} that category $\mathcal O_c$ can be viewed as a quasi-hereditary cover of the category of representations of the Hecke algebra (attached to any complex reflection group) while our construction yields an explicit functor from category $\mathcal O_c$ to the cyclotomic $q$-Schur algebra which is known to be a quasi-hereditary cover of the Hecke algebra.

We begin by recalling a construction of  the algebra $\mathcal H_\kappa$. Let $\mu_l$ denote the cyclic group of order $l$ (thought of as the $l$-th roots of unity in $\mathbb C$, \textit{i.e.} as a complex reflection group with a single ``hyperplane'' $\{0\}$). The action of $\mu_l$ on $\mathbb C$ induces an action on $\mathcal D(\mathbb C^\times)$, so we may consider the smash product $\mathcal D(\mathbb C^\times)\rtimes \mathbb C[\mu_l]$ of $\mathcal D(\mathbb C^\times)$ with the group algebra of $\mu_l$. To clarify the notation we fix a generator $\gamma$ of $\mu_l$ and write elements of $\mathbb C[\mu_l]$ in the form $\sum_{i=0}^{l-1} a_i\gamma^i$ where $a_i \in \mathbb C$ ($0\leq i \leq l-1$). Then if $\kappa = (\kappa_0,\ldots,\kappa_{l-1}) \in \mathbb C^l$ we define a Dunkl operator by the formula
\[
\partial_{\kappa} = \frac{d}{dz} + \frac{l}{z}\sum_{i=0}^{l-1} \kappa_i e_i,
\]
where the $e_i$ denote the idempotents of $\mathbb C[\mu_l]$, that is $e_i = (1/l)\sum_{j=0}^{l-1} \zeta^{ij}\gamma^j$, for $i=0,2,\ldots,l-1$. The algebra $\mathcal H_\kappa$ is then the subalgebra of $\mathcal D(\mathbb C^\times)\rtimes\mathbb C[\mu_l]$ generated by $z, \partial_\kappa$ and $\gamma$.

\subsection{}
We now review the description of  the minimal resolution $X$ of $\mathbb C^2/\mu_l$ as a hypertoric variety. Let $Q$ be the cyclic quiver with $l$ vertices $I = \{0,1,\ldots,l-1\}$ (read as elements of $\mathbb Z/l\mathbb Z$), and edges $H = \{h_i: 1\leq i \leq l\}$, where $h_i$ is the edge $(i-1) \to i$. Let $\bar{Q}$ denote the quiver $(I,E)$ where $E = H \cup \bar{H}$ where $\bar{H} = \{\bar{h}_i : 1\leq i \leq l\}$ and $\bar{h}_i$ denotes the edge from $i \to (i-1)$, \textit{i.e.} the edge in the opposite direction to $h_i$. We will denote the initial and terminal vertex of an edge $h \in E$ by $s(h)$ and $t(h)$ respectively.

A representation of $\bar{Q}$ is an $I$-graded vector space $V= \bigoplus_{i=0}^{l-1} V_i$ together with an element of $E_{V,\bar{H}} = \bigoplus_{h \in \bar{H}} \text{Hom}(V_{s(h)}, V_{t(h)})$, so that the moduli space of representations of $\bar{Q}$ (of a fixed graded dimension) may be identified with the quotient of $E_{V,\bar{H}}$ by the action of $G_V = \prod_{i \in I} \text{GL}(V_i)/\mathbb G_m$ where $\mathbb G_m$ is the diagonal copy of scalar operators (which acts trivially on $E_{V,\bar{H}}$). Similarly, we have the moduli space of representations of $Q$ given by the quotient of $G_V$ acting on $E_{V,H}$.

\subsection{}
We will be interested in the case where $V$ has graded dimension $\delta = (1,1,\ldots,1)$, and thus we assume $\dim(V_i)=1$ in what follows. We will also write $G = G_V \cong (\mathbb G_m)^l/D$ (where $D$ is the diagonal $\mathbb G_m$), a rank $l-1$ torus, with Lie algebra $\mathfrak g = \mathbb C^l/\mathbb C$. Note that character lattice $X(G)$ of $G$ may be viewed as the set of $\tht = (\tht_0,\tht_2,\ldots,\tht_{l-1}) \in \mathbb Z^l$ such that $\sum_{i=0}^{l-1}\tht_i =0$.

\begin{definition}
Given $\tht \in X(G)$, we say   a representation $(V,x)$ of $\bar{Q}$ is $\tht$-\textit{semistable} if for any subrepresentation $W$ of $V$ we have $\sum_{i=0}^{l-1} \theta_i\dim(W_i) \leq 0$. We set $\tilde{X}_\tht$ to be the subset of $E_{V,\bar{H}}$ consisting of $\tht$-semistable representations, a $G$-invariant subset of $E_V$. 
\end{definition}

Now $E_V \cong T^*E_{V,H}$ since we may identify $\text{Hom}(V_{i-1},V_{i})$ with the dual of $\text{Hom}(V_i, V_{i-1})$ via the trace pairing $(x,y) \mapsto \text{tr}_{|V_i}(yx)$. Therefore $E_{V,\bar{H}}$ is a symplectic vector space, and the action of $G$ clearly preserves the symplectic form. We will write $(a_i,b_i)_{1\leq i\leq l}$ for a point in $E_{V, \bar{H}}$ where $a_i \in \text{Hom}(V_{i-1},V_i)$ and $b_i\in \text{Hom}(V_{i},V_{i-1})$. Since $\tilde{X}_\tht$ is clearly a $G$-invariant open subset of $E_{V,\bar{H}}$, it is a complex symplectic manifold acted on by $G$. The associated moment map $\mu\colon \tilde{X}_\tht\to \mathfrak g$ is given by:
\[
(a_i,b_i)_{1 \leq i \leq l} \mapsto (a_{i+1}b_{i+1} - a_ib_i)_{1\leq i \leq l}.
\]

Let $\mu_l$ act on $\mathbb C^2$ via the action $\gamma(x,y) = (\gamma x,\gamma^{-1}y)$ for $\gamma \in \mu_l$, and write $\mathbb C^2/\mu_l$ for the quotient, a rational surface singularity of type $A$. Given $(a_i, b_i)_{1 \leq i\leq l} \in E_{V,\bar{H}}$, a pair $(\bar{a},\bar{b}) \in \mathbb C^2$ satisfying the conditions $\bar{a}^l = a_1a_2\ldots a_l$ and $\bar{b}^l = b_1b_2\ldots b_l$ and $\bar{a}\bar{b} = a_1b_1$ is well defined up to the action of $\mu_l$. This yields a map from $E_{V,\bar{H}}$ to $\mathbb C^2/\mu_l$ which induces an isomorphism of varieties between $E_{V,\bar{H}}//G $ and $\mathbb C^2/\mu_l$.

The G.I.T. quotient of $\tilde{X}_\theta$ by the action of $G$ is denoted $X_\theta$ (thus in particular $X_0 = E_{V,\bar{H}}//G$). We write $[a_i,b_i]_{1 \leq i\leq l}$ for the equivalence class in $X_\tht$ of  a point $(a_i,b_i)_{1 \leq i \leq l}$ in $\tilde{X}_\tht$. When $(\theta_i)$ satisfies $\sum_{k=i}^{j-1} \theta_k \neq 0$ for all  $i,j \in \mathbb Z/l\mathbb Z$ where $i \neq j$, this quotient is a smooth surface, and the map $\pi$ to the affine quotient $E_{V,\bar{H}}//G$ is a resolution of singularities, in fact the minimal resolution. 
From now on we will assume this condition on $\theta$ holds.

\subsection{}
The varieties $X_0$ and $X_\tht$ are both toric, where the action of $T = \mathbb G_m^2$ is induced by the action of $T$ on $E_{V,\bar{H}}$ given by $z.(a_i,b_i)_{1 \leq i \leq l} = (z_1a_i,z_2b_i)_{1 \leq i \leq l}$ where $z = (z_1,z_2) \in T$. To describe the structure of $X_\theta$ as a toric variety, we need to introduce an ordering on the elements of $\mathbb Z/l\mathbb Z$ induced by the stability $\theta$. 

\begin{definition}
We say that $i \triangleright j$ if $\theta_i + \tht_{i+1} + \ldots \tht_{j-1} < 0$. By our condition on $\tht$ this is a total order, and we write $\eta_i$ for the permutation of $\{1,2,\ldots, l\}$ which this yields, i.e. $\eta_1 \triangleright \eta_2 \triangleright \ldots \triangleright \eta_l$. 
\end{definition}

\noindent
As a toric variety $X_\tht$ can be written naturally as a union $X_\tht = \bigcup_{i=1}^l X_i$ where 
\[
X_i = \{[a_j,b_j]_{1 \leq j \leq l} : a_{\eta_j} \neq 0 \text{ for } j<i; b_{\eta_j} \neq 0 \text{ for } j>i\}
\]
and each $X_i$ is an affine toric variety containing a unique $T$-fixed point $p_i$ with $a_{\eta_i} = b_{\eta_i} = 0$. In fact $X_i \cong T^*\mathbb C$, which we may see explicitly as follows: if we pick a basis for the lines $(V_i)_{0 \leq i \leq l-1}$ then $E_{V,\bar{H}}$ may be identified with $\mathbb C^{2l}$ where we will still write a point of $\tilde{X}_\tht$ as $(a_i,b_i)_{1 \leq i \leq l}$. Let $\bar{x}_i((a_j,b_j)_{1 \leq j \leq l}) = a_i$ and $\bar{y}_i((a_j,b_j)_{1 \leq j \leq l}) = b_i$ be the obvious coordinate functions. Then $X_i = \text{Spec}(R_i)$ where $R_i$ is the subalgebra of function field $\mathbb C(\bar{x}_i,\bar{y}_i)_{1\leq i \leq l}$ generated by the algebraically independent elements
\[
\bar{f}_i = \frac{\bar{x}_{\eta_1} \bar{x}_{\eta_2} \ldots \bar{x}_{\eta_i}}{\bar{y}_{\eta_{i+1}}\bar{y}_{\eta_{i+2}} \ldots \bar{y}_{\eta_l}}, \quad \bar{g}_i = \frac{\bar{y}_{\eta_i}\bar{y}_{\eta_{i+1}} \ldots \bar{y}_{\eta_l}}{\bar{x}_{\eta_1} \bar{x}_{\eta_2} \ldots \bar{x}_{\eta_{i-1}}}.
\]
Note that $\bar{f}_i\bar{g}_i = \bar{x}_{\eta_i}\bar{y}_{\eta_i}$. The subvarieties $X_i$ and $X_{i+1}$ intersect in a complex line whose closure in $X_\tht$ we denote by $D_i$. For $1 \leq i <l$ this closure is a $\mathbb P^1$ and their union is the exceptional divisor of the resolution $\pi$. If we let $\mathcal L$ denote the preimage of $\pi^{-1}(\{(\bar{a},\bar{b}) \in X_0: \bar{a}\bar{b} = 0\})$ then $\mathcal L$ is the union of $l+1$ irreducible components $D_i$, ($0 \leq i \leq l$), where $D_0 \cong D_l \cong \mathbb C$, and the remaining $D_i$ are the above $\mathbb P^1$s. These $D_i$ are all the $T$-stable divisors in $X_\tht$.

\subsection{}
\label{Walgconstruction}
We now recall the construction of $\mathcal W$-algebras with an $F$-action on $X_\tht$ following \cite{KR}. Let $c = (c_i)_{0 \leq i \leq l-1}$ be an $l$-tuple of complex numbers such that $\sum_{i=0}^{l-1}c_i= 0$. The $\mathscr W$-algebra $\mathscr A_c$ is obtained by quantum Hamiltonian reduction from the standard $\mathscr W$-algebra on $\mathbb C^{2l} = T^*\mathbb C^n$, whose generators we denote by $x_i,y_i$ ($1 \leq i \leq l$). Clearly $\tilde{X}_\tht$ carries the sheaf of algebras $\mathscr W_{\tilde{X}_\tht}$, the restriction of the sheaf  $\mathscr W_{T^*\mathbb C^n}$, and we may define the $\mathscr W_{\tilde{X}_\tht}$-module $\mathscr L_c$ by 
\[
\mathscr L_c = \mathscr W_{\tilde{X}_\tht}/\big(\sum_{i=1}^l \mathscr W_{\tilde{X}_\tht}(x_{i+1}y_{i+1} - x_iy_i + \hbar c_i)\big).
\]

If $p \colon \mu^{-1}(0)\to X_\tht$ denotes the quotient map, then $\mathscr A_c = (p_*(\mathcal{E}nd_{\mathscr W_{\tilde{X}_\tht}}(\mathscr L_c))^G)^{\text{op}}$ is a $\mathscr W$-algebra on $X_\tht$. The sheaf of algebras $\Ac$ obtained by adjoining a square root of $\hbar$ has an $F$-action in the sense of \cite[2.3]{KR} of weight one for the $\mathbb C^\times$ action giving $x_i$ and $y_i$ weight one and $\hbar$ weight two. We then consider the category $\text{Mod}_F^{\text{good}}(\tilde{\mathscr A}_c)$ of good $\tilde{\mathscr A}_c$-modules with an $F$-action as in \cite[2.4]{KR}. It is known by the work of Holland \cite{H} and Bellamy-Kuwabara \cite[\S 6]{BK} that the algebra $A_c = \text{End}_{\text{Mod}_F^{\text{good}}(\tilde{\mathscr A}_c)}(\tilde{\mathscr A}_c)^{\text{opp}}$ is isomorphic to the spherical rational Cherednik algebra $U_\kappa = e_0H_{\kappa}e_0$ (where $e_0$ is the idempotent corresponding to the trivial representation of $\mu_\ell$) and moreover \cite{BK} proves a localization theorem if $i\triangleright j$ whenever $c_i+c_{i+1} +\ldots c_{j-1} \in \mathbb Z_{\geq 0}$ (though we do not need this for our constructions, which only use the existence of the localization functor). Here the deformation parameter $c = c(\kappa)$ is related to the parameter $\kappa$ via 
\[
c_i = \kappa_{i+1} - \kappa_i +1/l -\delta_{i,0}.
\]

\subsection{}
Let $\mathcal O_\kappa$ denote the category of representations of the spherical rational Cherednik algebra on which $\partial_\kappa$ acts locally nilpotently, and let $\mathcal O_\kappa^{\text{sph}}$ denote the corresponding category of representations for $U_\kappa$. Via localization, this category corresponds to a category of modules for $\Ac$, which we will denote by $\mathcal O_c$. It follows from results of \cite[\S 5]{Ku} that the modules in $\mathcal O_c$ are all supported on the union of divisors $\bigcup_{i=1}^l D_i$ (since the results there give a description of the supports of all simple modules in $\mathcal O_c$). 

\begin{definition}
The cyclotomic Hecke algebra $\mathcal K_l$ for $G(1,1,l)$ is just the commutative algebra $\mathbb C[T]$ (where $T$ is the element conventionally denoted by $T_0$ in the literature) subject to the relation 
\[
\prod_{i=1}^m(T - q^{a_i})=0
\]
where $(a_i)_{i=1}^l \in \mathbb Z^l$, and $q \in \mathbb C^\times$. 
\end{definition}

To any cyclotomic Hecke algebra one can attach a cyclotomic $q$-Schur algebra \cite{DJM}. We now recall how it is defined (in our special case). For $i \in \{1,2,\ldots, l\}$ let 
\[
m_i = \prod_{i<j\leq l} (T - q^{a_j}).
\]
and let $M^i$ be the $\mathcal K_l$-ideal $\mathcal K_l m_i$. Note that if we let $r_i = \prod_{j=1}^i (T-q^{a_j})$ then $M^i \cong \mathbb C[T]/(r_i)$ via the map $f \mapsto f.m_i$. The cyclotomic $q$-Schur algebra is defined to be 
\[
\mathcal S_l = \text{End}_{\mathcal K_l}(\bigoplus_{i=1}^l M^i) = \bigoplus_{i.j \in [1,l]} \text{Hom}_{\mathcal K_l}(M^i,M^j).
\]
\noindent
The following simple lemma yields an explicit description of this algebra, by noting that any $\mathcal K_l$-module can be thought of as a $\mathbb C[t]$-module where $t$ is an indeterminate via the surjection $\mathbb C[t] \to \mathcal K_l$ given by $t \mapsto T$.

\begin{lemma} Let $t$ be an indeterminate, and let $g_1$ and $g_2 \in \mathbb C[t]$. Then we have
\[
\text{Hom}_{\mathbb C[t]}(\mathbb C[t]/(g_1),\mathbb C[t]/(g_2)) \cong \mathbb C[t]/(g),
\]
where $g = \text{g.c.d.}(g_1,g_2)$.
\end{lemma}
\begin{proof}
If $\phi \in \text{Hom}_{\mathbb C[t]}(\mathbb C[t]/(g_1),\mathbb C[t]/(g_2))$ then $\phi(1+(g_1)) = u \in \mathbb C[t]/(g_2)$, where $g_1.u =0 \in \mathbb C[t]/(g_2)$. The submodule of $\mathbb C[t]/(g_2)$ consisting of such elements is isomorphic as a $\mathbb C[t]$-module to $\mathbb C[t]/(\text{g.c.d.}(g_1,g_2))$, and the map $\phi \mapsto u$ gives the required isomorphism.
\end{proof}

\noindent
Let $R_i = \text{End}_{\mathcal K_l}(M^i)\cong \mathbb C[t]/(r_i)$ (by, for example, the previous Lemma). If we write $\mathcal S_l^{ij} = \text{Hom}_{\mathcal K_l}(M^j, M^i)$ then the Lemma shows that  for $i\leq j$ we may view $\text{Hom}(M^j,M^i)$ as a free $\mathbb C[t]/(r_i)$-module generated by the quotient map $\pi_{ij}\colon \mathbb C[t]/(r_j) \to \mathbb C[t]/(r_i)$ whereas if $i \geq j$ we may view $\mathcal S_l^{ji}$ as a free $\mathbb C[t]/(r_j)$-module with generator $m_{ji} = \prod_{i<k\leq j}(t-q^{a_k}) \in \mathbb C[t]/(r_i)$, corresponding to the inclusion $\iota_{ij} \colon \mathcal K_lm_j \hookrightarrow \mathcal K_l m_i$.

 We will thus write an element of $\mathcal S_l$ as $(a_{ij})_{i,j \in [1,l]}$ where $a_{ij} \in \mathcal S_l^{ij}$ is as an element of a free $R_{\text{min}(i,j)}$-module with generator $\pi_{ji}$ or $\iota_{ji}$ according to whether $i\leq j$ or $i\geq j$ (where we have $\pi_{i,i} = \iota_{i,i} = \text{id}_{M^i}$).  If the $q^{a_i}$s are all distinct, it is straightforward to check that this algebra splits into a direct sum of $l$ matrix algebras of dimensions $1,2,\ldots,l$.
 
\begin{definition}
\label{Zlstandards}
The algebra $\mathcal S_l$ has a natural family of modules known as \textit{standard} or \textit{Weyl} modules. For each $i,j$ the algebra $R_i$ has an ideal $I_j^i$ generated by the element 
\[
s_j^i = \prod_{1 \leq k \leq i; k \neq j}(T-q^{a_k}).
\]
If $j\leq i$ this is a one-dimensional $\mathbb C$-vector space, isomorphic to $\mathbb C[T]/(T-q^{a_j})$ as an $R_i$ module, otherwise it is zero. For each $j$, ($1 \leq j \leq l$) let $W_j = \bigoplus_{k=1}^l I_j^k = \bigoplus_{k =j}^l I_j^k$.  We may give $W_j$ the structure of a $\mathcal S_l$-module as follows: if $a = (a_{ij})_{1 \leq i,j \leq l} \in \mathcal S_l$ and $n = (n_k)_{j \leq k \leq l}$ where $a_{ij} \in \mathcal S_l^{ij}$, and $n_k \in I_j^k$ then 
\[
a.n = (\sum_{r} a_{kr}n_r)_{j \leq k \leq l},
\]
where $a_{kr}n_r$ denotes the image of $n_r$ under the chain of morphisms:

\xymatrix{
& & & & R_r \ar[r]^{\cong} & M^r \ar[r]^{a_{kr}} & M^k \ar[r]^\cong & R_k.
}
\noindent
It is easy to see that under this composition, $I_j^r$ maps to $I_j^k$ if $j \leq k$ so that the formula does indeed give an $\mathcal S_l$-module structure. Note that in the situation where the $q^{a_k}$ are all distinct, the standard modules are clearly precisely the simple modules.
\end{definition}

\begin{example}
If $n=2$ then the algebra $\mathcal S_l$ can be viewed as the set of matrices of the form
\[
\left(\begin{array}{cc}a_1 & a_2\pi \\a_3\varepsilon & b \end{array}\right)
\]
where we write $\pi$ for $\pi_{21}$ and $\varepsilon = m_{21}$, and the $a_i$ are in $R_1= \mathbb C$, and $b \in R_2$. 

If $q^{a_1} = q^{a_2}$ then setting $\varepsilon = T-q^{a_1}$ we have $R_2 = \mathbb C[\varepsilon]$ where $\varepsilon^2=0$. Using the basis $\{1,\varepsilon\}$ for $R_2$ the multiplication in $\mathcal S_2$  then becomes:
\[
\left(\begin{array}{cc}a_1 & a_2\pi \\a_3\varepsilon & b_1 + b_2\varepsilon \end{array}\right) \left(\begin{array}{cc}a'_1 & a'_2\pi\\a'_3\varepsilon & b'_1 + b'_2\varepsilon \end{array}\right) = \left(\begin{array}{cc}a_1a'_1 & (a_1a'_2 + a_2b'_1)\pi \\(a_3a'_1 + b_1a'_3)\varepsilon & b_1b'_1 + \varepsilon(a_3a'_2b_1b'_2 + b'_2b_1) \end{array}\right).
\]
It is then easy to see that the subspace of elements of the form $\left(\begin{array}{cc}0 & 0\\a_3\varepsilon & b_2\varepsilon \end{array}\right)$ forms a two-sided ideal, with the quotient algebra isomorphic to the algebra of $2$-by-$2$ upper-triangular matrices.
\end{example}

Let $\rho_i \in \mathcal S_l$ denote the projection $\rho_i \colon \bigoplus_{j=1}^l M^j \to M^i$. We now wish to find a presentation for $\mathcal S_l$. 

\begin{definition}
Let $(a_i)_{i=1}^l \in \mathbb C^l$ and let $\mathcal S'_l$ be the algebra given by generators $T,\rho_i,\pi_{i,i+1},\iota_{i+1,i}$ subject to relations:
\begin{enumerate}
\item
$\rho_i\rho_j = \delta_{ij}\rho_i, T\rho_i = \rho_iT,\quad (1\leq i,j \leq l)$;
\item
$\sum_{i=1}^l \rho_i = 1$;
\item
$\rho_i\pi_{i,i+1} = \pi_{i,i+1}\circ\rho_{i+1} = \pi_{i,i+1},  \quad (1 \leq i \leq l-1)$;
\item
$\rho_{i+1}\iota_{i+1,i} = \iota_{i+1,i}\circ\rho_i = \iota_{i+1,i}, \quad (1 \leq i \leq l-1)$;
\item 
$\iota_{i+1,i} \circ \pi_{i,i+1} = (T-q^{a_{i+1}})\rho_{i+1}, \quad (1 \leq i \leq l-1)$; 
\item
$\pi_{i,i+1}\circ\iota_{i+1,i} = (T-q^{a_{i+1}})\rho_i, \quad (1\leq i \leq l-1)$;
\item
$T\rho_1 = q^{a_1}\rho_1$.
\end{enumerate}
It is easy to see that the generator $T$ of $\mathcal S'_l$ is central, and in fact redundant. (We include it for convenience which will become evident below.)
\end{definition}

The next elementary lemma shows moreover that the spectrum of each $T\rho_i$ is very constrained.
\begin{lemma}
\label{spectrum}
Let $\mathscr A$ denote an associative $\mathbb C(q)$-algebra with elements $A$ and $B$, and suppose we set $T_i = AB+q^{a}$ and $T_{i+1} = BA+q^{a}$ for some $a \in \mathbb C$. If $T_i$ satisfies the equation $P(T_i) = \prod_{k=0}^{i-1}(T_i - q^{a_k}) = 0$, then $T_{i+1}$ satisfies the equation
\[
\prod_{k=0}^{i}(T_{i+1}-q^{a_k}) =0.
\]
where we set $a_i = a$.
\end{lemma}
\begin{proof}
Clearly we have $Q(AB) = 0$, where $Q(t+q^a) = P(t)$. Now write $Q(t) = \sum_{k=0}^{i-1}b_kt^k$. Then we have
\[
\begin{split}
0=B.Q(AB).A &= \sum_{k=0}^{i-1}b_kB.(AB)^k.A,\\
&= \sum_{k=0}^{i-1} b_k(BA)^{k+1} = (BA)Q(BA)\\
&= (T_{i+1}-q^a)\prod_{k=0}^{i-1}(T_{i+1}-q^{a_k}),
 \end{split}
\] 
since $Q(BA) = P(T_{i+1})$, and thus $T_{i+1}$ satisfies the required equation.
\end{proof}

\begin{cor}
The algebra $\mathcal S'_l$ is isomorphic to the algebra $\mathcal S_l$.
\end{cor}
\begin{proof}
It is clear that there is a map $p:\mathcal S'_l \to \mathcal S_l$, sending $\pi_{i,i+1}$ and $\iota_{i+1,i}$ to their corresponding elements in $\mathcal S_l$ and $\rho_i$ to $\text{id}_{M^i}$. Since $\mathcal S_l$ is generated by the $\pi_{i, i+1}$, $\iota_{i+1,i}$ and $\text{id}_{M^i}$s it is clear that $p$ is surjective. For $i<j$ let $\pi_{ij} = \pi_{i,i+1}\circ\ldots \circ \pi_{j-1,j} \in \mathcal S'_l$, and similarly if $i>j$ let $\iota_{ij} = \iota_{i,i-1}\circ\ldots\circ\iota_{j+1,j}$. Then using the defining relations, it is easy to see that as a $\mathbb C[T]$-module $\mathcal S'_l$ is spanned by the elements $\{\rho_i, \pi_{jk}, \iota_{kj}: 1\leq i \leq l, 1\leq j<k \leq l\}$. Applying Lemma \ref{spectrum} with $T_i = \rho_iT$ and using induction it is easy to see that $\mathbb C[T]$ acts on $\mathbb C[T]\iota_{kj}$ and $\mathbb C[T]\pi_{jk}$ via the algebra $\mathbb C[t]/(r_j)$, and on $\mathbb C[T]\rho_{i}$ via the algebra $\mathbb C[t]/(r_i)$. Comparing with matrix-like the description of $\mathcal S_l$ we obtained above it is clear that $\dim_{\mathbb C}(\mathcal S'_l) \leq \dim_{\mathbb C}(\mathcal S_l)$, so that $p$ must be an isomorphism.
\end{proof}

\begin{remark}
\label{modulestructure}
It follows that a vector space $V$ can be equipped with an $\mathcal S_l$-module structure by giving a grading $V = \bigoplus_{i=1}^l V_i$ together with maps $\alpha_i\colon V_i \to V_{i+1}$ and $\beta_i\colon V_{i+1} \to  V_i$, where we have 
\begin{equation}
\label{relationsonSl}
\alpha_{i-1}\circ\beta_{i-1} - \beta_i\circ\alpha_i = (q^{a_{i+1}} - q^{a_i})\text{id}_{V_i}, \quad 1\leq i \leq l-1,
\end{equation}
(where we set $\alpha_0=\beta_0=0$). 
\end{remark}

\subsection{}
We now show that our microlocal approach for the representations of category $\mathcal O_\kappa$ of the rational Cherednik algebra of type $\mu_l$ gives a functor to the category of finite dimensional representations of $\mathcal S_l$. In fact, using the work of \cite{BK} and \cite{Ku} it follows this functor is an equivalence (provided $\kappa$ lies outside certain explicit hyperplanes). We use the nearby and vanishing cycle construction for $\mathcal D$-modules on $\mathbb C$ which are $\mathcal O$-coherent along $\mathbb C^\times$ as described in $\S$\ref{nearbyvanishing}. Topologically, this can be viewed as follows: for a $\mathcal D$-module $\mathcal M$, the vector space $\Psi(\mathcal M)$ corresponds to the stalk of the solution sheaf at a point in the punctured disk, (equipped with the natural monodromy automorphism) and the vector space $\Phi(\mathcal M)$ gives a ``Morse group'' of the solution sheaf of $\mathcal M$ at a generic covector in the cotangent space $T^*_0\mathbb C$, which comes equipped with a microlocal monodromy operator.

Let $\Aci$ denote the restriction of $\Ac$ to $X_i \cong T^*\mathbb C$. It is shown in \cite[\S 3]{Ku} that $\Aci$ is isomorphic to the standard $\mathscr W$-algebra on $T^* \mathbb C$ via the map defined by $x \mapsto f_i$ and $\xi \mapsto g_i$ where we set $f_i = (x_{\eta_1}\ldots x_{\eta_i})\circ(y_{\eta_{i+1}}\ldots y_{\eta_l})^{-1}$ and $g_i = (y_{\eta_i}\ldots y_{\eta_l})\circ (x_{\eta_1}\ldots x_{\eta_{i-1}})^{-1}$. Note that $f_i \circ g_i = x_{\eta_i}\circ y_{\eta_i}$ and $g_i\circ f_i = x_{\eta_i}\circ y_{\eta_i}+ \hbar$. Via this isomorphism, the $F$-action on $\tilde{\mathscr A}_{c|X_i}$ corresponds to the $F$-action on the standard $\mathscr W_{T^*\mathbb C}$-algebra given by $x \mapsto t^{2i-l}x$ and $\xi\mapsto t^{l-2i+2}$, so that the $F$-invariant sections are then:
\[
\text{End}_{\text{Mod}_F(\mathscr W[\hbar^{1/2}])}(\mathscr W[\hbar^{1/2}])^{\text{opp}} = \mathbb C[\hbar^{l/2-i}x, \hbar^{i-l/2}\xi],
\]
which is isomorphic to $\mathcal D(\mathbb C)$, and moreover the category of modules $\text{Mod}_F(\mathscr W[\hbar^{1/2}])$ is equivalent to $\text{Mod}_{\text{coh}}(\mathcal D(\mathbb C))$ and hence $\text{Mod}_{\text{coh}}(\mathcal D_\mathbb C)$ (see the second example of \cite[2.3.3]{KR} for more details, where in our case, $m=2$). Note that the element $x\partial$ in $\mathcal D(\mathbb C)$ corresponds to $\hbar^{-1}f_ig_i$.

\begin{definition}
For $i \in \{1,2,\ldots, l\}$ we define functors $KZ_i$ from $\mathcal O_c$ to $\mathbb C[t]_{(t)}$-mod as follows. Let $\mathcal M_i$ denote the restriction of $\mathcal M$ to $X_i$, where we may view it (by the above discussion) as a module for $\mathscr W_{T^*\mathbb C}$, with the appropriate $F$-action, and hence as a $\mathcal D_\mathbb C$-module. As such, it corresponds to a holonomic module whose support lies in $\{(x,\xi) : x\xi = 0\}$, where $\{x=0\}$ corresponds to $D_{i-1}$ and $\{\xi =0\}$ corresponds to $D_i$. Therefore it yields a local system on $D_i \backslash \{p_i\} \cong \mathbb C^\times$, which is the same data as a vector space equipped with an automorphism, which we may view as a $\mathbb C[t]_{(t)}$-module. We set $KZ_i(\mathcal M)$ to be this $\mathbb C[t]_{(t)}$-module, $(\Phi(\mathcal M_i),T)$.
\end{definition}

We now show how the functors $KZ_i$ yield a representation of $\mathcal S_l$, where $\mathcal S_l$ is the algebra with parameters $(q^{a_i})_{i=1}^l$ where $a_i = \sum_{j=1}^{i-1} \bar{c}_j$ (so that $a_0 =0$) and for $a\in \mathbb C$ we write $q^a$ for $\text{exp}(2i\pi a)$. Indeed given $\mathcal M$ an object in $\mathcal O_c$, define 
\[
\mu KZ(\mathcal M) = \bigoplus_{i=1}^l KZ_i(\mathcal M).
\]
Thus $\mu KZ(\mathcal M)$ is a graded vector space.

\begin{prop}
The $\mathbb Z/l\mathbb Z$-graded vector space $\mu KZ(\mathcal M)$ has a natural $\mathcal S_l$-module structure, so that $\mu KZ$ becomes a functor from $\mathcal O_c$ to $\mathcal S_l$-mod.
\end{prop}
\begin{proof}
To equip $\mu KZ(\mathcal M)$ with the structure of a $\mathcal S_l$-module, by Remark \ref{modulestructure} we need only define appropriate maps 
\[\alpha_i\colon KZ_i(\mathcal M) \to KZ_{i+1}(\mathcal M); \quad \beta_i \colon KZ_{i+1}(\mathcal M) \to KZ_i(\mathcal M).
\]
For these we use (slightly rescaled versions of) the natural transformations $v$ and $c$ between the nearby and vanishing cycle functors. Let $var_i\colon \Psi(\mathcal M_i) \to \Phi(\mathcal M_i)$ and $can_i \colon \Phi(\mathcal M_i) \to \Psi(\mathcal M_i)$ denote these morphisms, where as above $\mathcal M_i = \mathcal M_{|X_i}$, and note that $KZ_i(\mathcal M) = \Phi(\mathcal M_i) = \Psi(\mathcal M_{i+1})$.

Set $\alpha_i = q^{a_i}var_{i+1}$ and $\beta_i = can_{i+1}$. We need only verify that Equation \ref{relationsonSl} holds. For this, note that if $T_i$ denotes the monodromy automorphism on $KZ_i(\mathcal M)$, then $T_i = \text{exp}(-2i\pi \hbar^{-1}f_ig_i) = var_ican_i+1$.  Since on $X_i \cap X_{i+1}$ we have $f_ig_i = x_{\eta_i}\circ y_{\eta_i} = x_{\eta_{i+1}}\circ y_{\eta_{i+1}} + \hbar \bar{c}_i$, it follows that $T_i = q^{\bar{c}_i}T_{i+1}$ on $KZ_i(\mathcal M) = \Phi(\mathcal M_i) = \Psi(\mathcal M_{i+1})$. Then we have
\[
\begin{split}
\alpha_{i-1}\circ\beta_{i-1} -\beta_i\circ \alpha_i &= q^{a_i}(T_i-1) -q^{a_{i+1}}(T_{i+1}-1) \\
&= q^{a_i}(q^{\bar{c}_i}T_{i+1}-1) -q^{a_{i+1}}(T_{i+1} -1)\\
&= q^{a_{i+1}} - q^{a_i},
\end{split}
\]
as required. Finally, note that on $X_1$ the module $\mathcal M_1$ is supported entirely on $D_1$, so that $\Psi(\mathcal M_1) = 0$, and hence the above calculation also shows that $\beta_1\circ\alpha_1 = q^{a_2} - q^{a_1}$ and we are done. 
\end{proof}

\begin{cor}
Let 
\[
s_i(t) = \prod_{j=1}^{i} (t-q^{\sum_{k=j}^{i-1}\bar{c}_j}).
\]
The functor $KZ_i$ from $\mathcal O_c$ to $\mathbb C[t]_{(t)}$ factors through $\mathbb C[t]/(s_i)$.
\end{cor}
\begin{proof}
If $T$ is the central element of $\mathcal S_l$, then the action of the monodromy operator on $KZ_i(\mathcal M)$ (for $\mathcal M$ an object of $\mathcal O_c$) is given by $q^{\sum_{k=i}^l\bar{c}_k}\rho_iT$. The claim then follows from fact our description of the structure of $\mathcal S_l$.
\end{proof}

\begin{remark}
Note that in particular the functor $KZ_l$ therefore yields a representation of the corresponding cyclotomic Hecke algebra. In fact it is just the original $KZ$-functor for the rational Cherednik algebra of type $\mu_l$. This can be seen using the correspondence between the $\kappa_i$ and $c_i$s, and an argument analogous to Proposition \ref{KZcheck}. Using \cite{Ku} (where an explicit construction of the sheaves corresponding to standard modules is given), the reader can check that $\mu KZ$ sends the standard modules in $\mathcal O_c$ to the standard modules for $\mathcal S_l$, and similar arguments allow one to show that $\mu KZ$ is in fact an equivalence.
\end{remark}

\section{Appendix}
\subsection{Twisted $\mathcal D$-modules}
\label{twistingstuff}
In this paper we work with modules over rings of twisted differential operators on $\mathfrak{gl}_n \times \mathbb P^{n-1}$, so we recall briefly the construction of the twisted rings we use. For a detailed discussion of these issues see \cite{K7}, whose presentation we largely follow\footnote{As noted in \cite{K7}, the presentation there uses the notion of a twisting datum (recalled above) which is slightly \textit{ad hoc} from a systematic point of view, however it is well-suited to the families of twistings which we need.} (for an alternative account in the algebraic context see \cite{BB}). If $X$ is a topological space a \textit{twisting datum} $\tau$ is a triple $(\pi\colon X_0\to X, L,m)$ where $\pi\colon X_0 \to X$ is a continuous map admitting a section locally on $X$, $L$ is an invertible sheaf of vector spaces on $X_0\times_X X_0$ and $m$ is an isomorphism:
\[
m\colon p_{12}^{-1}L\otimes p_{23}^{-1}L \to p_{13}^{-1}L,
\]
on $X_2 = X_0\times_X X_0\times_X X_0$. Moreover, the isomorphism $m$ is required to satisfy an appropriate ``associativity'' condition on the quadruple product of $X_4$ of $X_0$ over $X$. Given a twisting datum, a twisted sheaf on $X$ is a pair $(F,\beta)$ consisting of a sheaf $F$ on $X_0$ and an isomorphism $\beta\colon L\otimes p_2^{-1}F \to p_1^{-1}F$, where $p_1,p_2$ are the natural projections from $X_1 = X_0\times_X X_0$ to $X_0$, along with the requirement that $\beta$ satisfies an appropriate cocycle condition.

If $H$ is a complex affine algebraic group and $\pi\colon X_0 \to X$ a principal $H$-bundle (where we assume that $H$ acts on $X_0$ on the left) then one can naturally attach to $X_0$ a family of twisting data on $X$.  To describe this we first need the notion of a character local system. Let $\mu\colon H\times H \to H$ be the multiplication map. A character local system is an invertible $\mathbb C_H$-sheaf $L$ equipped with an isomorphism $m\colon q_1^{-1}L \otimes q_2^{-1}L \to \mu^{-1}L$ which satisfies the associative law. Let $\mathfrak h = \text{Lie}(H)$ denotes the Lie algebra of $H$ and take an $H$-invariant element $\lambda$ of $\mathfrak h^* = \text{Hom}(\mathfrak h,\mathbb C)$, so that $\lambda([\mathfrak h,\mathfrak h]) = 0$. Let $L_\lambda$ be the sheaf of (analytic) functions on $H$ which satisfy $R_H(X)(f)=\lambda(X).f$, where $R_H$ denotes map from $\mathfrak h$ to vector fields on $H$ given by the right action of $H$ on itself. The multiplication map $H\times H \to H$ then induces the structure of a character local system on $L_\lambda$. 

Now suppose that in addition we have an $H$-space $Y$, and let $a\colon H\times Y\to Y$ be the action map, and $p_1\colon H \times Y \to Y$, $p_2\colon H\times Y\to H$ the obvious  projections. An $(H,\lambda)$-equivariant sheaf on $Y$, or $\lambda$-twisted equivariant sheaf, is a pair $(F,\beta)$ consisting of a sheaf $F$ on $Y$ together with an isomorphism $\beta\colon p_1^{-1}L_\lambda \otimes p_2^{-1}F \to a^{-1}F$, once again satisfying an appropriate associative condition. In the case when $\lambda=0$ we say that $F$ is an equivariant sheaf. The notion of a twisted equivariant sheaf if closely related to that of a twisted sheaf, as the following construction shows. 

Let $\pi\colon X_0\to X$ be a principal $H$-bundle, and let $\lambda$ be as above. We may identify $H\times X_0$ with $X_0\times_X X_0$ by the map $(h,x') \mapsto (hx',x')$ and via this isomorphism, we find that $p_1^{-1}L_\lambda$ yields a twisting datum $\tau_\lambda$ on $X$. In particular, the categories of $\tau_\lambda$-twisted sheaves on $X$ is equivalent to the category of $(H,\lambda)$-equivariant sheaves on $X_0$, and thus the principal bundle $X_0$ yields the family $\{\tau_\lambda: \lambda \in \mathfrak (h^*)^H\}$ of twisting data. 

Thus far our discussion has been purely topological, but the same formalism may be used with $\mathcal D$-modules as is explained in \cite[\S 7.11]{K7}. One replaces $L_\lambda$ with the sheaf $\mathcal L_\lambda = \mathcal D_H.u_\lambda$ where $u_\lambda$ has defining relations $R_H(A)u_\lambda = \lambda(A)u_\lambda$ for $A \in \mathfrak h$ (and pullbacks \text{etc.} of sheaves of vector spaces with their appropriate $\mathcal D$-module analogues). Note that $L_\lambda \cong \mathcal H\text{om}_{\mathcal D_X}(\mathcal L_\lambda,\mathcal O_H^{\text{an}})$, where $\mathcal O_H^\text{an}$ is the sheaf of analytic rather than algebraic functions on $H$, and in fact the Riemann-Hilbert correspondence extends to have a twisted analogue. The ring of twisted differential operators $\mathcal D_{X,\tau_\lambda}$ (or more simply $\mathcal D_{X,\lambda}$) on $X$ is then given as follows: if $\mathcal N_\lambda = \mathcal D_{X_0}v_\lambda$ is the $\mathcal D_{X_0}$-module with defining relations $L_{X_0}(A) = -\lambda(A).v_\lambda$, then
\[
\mathcal D_{X,\lambda} = \{f \in \pi_*(\mathcal{E}nd_{\mathcal D_{X_0}}(\mathcal N_\lambda)): f \text { is } H\text{-equivariant}\}^{\text{op}}
\]
where $(\cdot)^{\text{op}}$ denotes the opposite ring.  By \cite[Lemma 7.12.1]{K7}, the category of modules for the $\tau_\lambda$-twisted ring of differential operators $\mathcal D_{X,\lambda}$ is then equivalent to the category of $(H,\lambda)$-equivariant $\mathcal D_{X_0}$-modules on $X_0$. Note that in the terminology of \cite{K7}, $(H,\lambda)$-equivariant $\mathcal D$-modules form an abelian subcategory of the category of $H$-quasi-equivariant $\mathcal D$-modules.

Finally, we need to consider equivariant twisted $\mathcal D$-modules. Given an affine algebraic group $G$ and a $G$-space $X$, one can naturally define the notion of a $G$-equivariant twisting datum on $X$, and hence the notion of $G$-equivariant twisted sheaves on $X$. In particular, in the case where the twisting data arises from a principal $H$-bundle $X_0$, if the group $G$ acts on $X_0$ and $X$ in such a way that the map $\pi$ is $G$-equivariant, and the actions of $H$ and $G$ commute on $X_0$, then we may define $G$-equivariant twisted $\mathcal D_{X,\lambda}$-modules, and the equivalence between twisted modules on $X$ and $(H,\lambda)$-equivariant modules on $X_0$ allows us to identify such modules with $\mathcal D_{X_0}$-modules which are $G$-equivariant and $(H,\lambda)$-twisted equivariant.

\subsection{Index Theorems and Characteristic Cycles.}
We have computed the characteristic cycle of standard modules using the local Euler obstruction and the known calculation of the stalk cohomologies. In this appendix we briefly recall the index theorem of Kashiwara, Dubson, Brylinski \cite{K73}, \cite{BDK}, and in the topological setting \cite {M} which is the key to this approach. This theorem can be seen as a (local version of a) generalization of the classical Hopf index theorem which calculates the Euler characteristic of a compact manifold $X$ as the self-intersection number of $X$ in its cotangent bundle. 

Since we only need a local result, we may suppose that $X$ is a stratified analytic subset of affine space $\mathbb C^n$. Thus $X = \bigsqcup_{S \in \mathcal S} S$, where each $S$ is a smooth locally-closed connected subset of $\mathbb C^n$, and the closure of a stratum $S$ is a union of strata. We may also assume that  the Whitney $(a)$ and $(b)$ conditions are satisified (or for that matter the $\mu$-condition introducted by Kashiwara and Schapira \cite[Chapter 8]{KS}). For $x \in X$ we will write $B_\varepsilon(x)$ for the set
\[
\{y \in X: \|y-x\| < \varepsilon\},
\]
where $\|.\|$ is the standard Hermitian norm on $\mathbb C^n$. (Thus our constructions here use the analytic variety attached to the algebraic varieties we considered earlier). 

\begin{definition}
Let $z \in S$ be a point of the stratum $S$. By taking a normal slice $N$ to $S$ at $x$, that is, choosing a complex analytic submanifold $N$ which intersects each stratum transversely and such that $N \cap S = \{z\}$, we may reduce to the case $S = \{z\}$. Let $\phi \colon N \to \mathbb C$ be a holomorphic function vanishing at $z$ such that $d\phi(z)$ does not lie in the closure of $T^*_TX$ for any stratum $T \neq S$. If $S=T$ we define $c_{S,S} =1$ for all $S$.

Endow $X$ with a Hermitian metric (say by taking the restriction of the standard on on $\mathbb C^n$) and pick a small disk $B = B_\varepsilon(z)$ about $z$, and a generic $\eta \in \mathbb C^\times$ such that $|\eta|<< \varepsilon$. The complex link $L$ of the stratum $S$ in $T$ is then defined to be the set 
\[
L = B\cap T\cap N \cap \phi^{-1}(\eta).
\]
Stratified Morse theory shows that the homeomorphism type of $L$ is independent of the choices made (in fact it is independent of the metric $\|.\|$ also, so does not depend on the choice of local embedding we make). We define the local Euler Obstruction $c_{S,T}$ to be the Euler characteristic of the complex link, that is
\[
c_{S,T} = \chi_c(L).
\]
As the notation suggests, this number is independent of the choice of $z$ in the stratum $S$ (here we use the assumption that our strata are connected).

\end{definition}

The index theorem shows that the characteristic cycle of a holonomic $\mathcal D$-module determines its local Euler characteristics. Indeed suppose that $M$ is a $\mathcal D$-module whose characteristic cycle lies in $\bigsqcup_{S \in \mathcal S} T_S^*M$ (so in particular $M$ is holonomic). Then since we assume that our stratification satisfies the Whitney conditions, $\text{DR}(M)$ is a constructible sheaf which is locally constant on the strata $S$, and we may set
\[
\chi_S(M) = \sum_{i}(-1)^i \dim(\mathcal H^i(\text{DR}(M))_{|S}).
\]
We also have $CC(M) = \sum_{S \in \mathcal S} m_S(M)[T^*_SX]$. 

\begin{Theorem}
\label{indextheorem}
Let $M$ be a holonomic $\mathcal D$-module as above. Then we have
\[
m_S(M) = \sum_{T \subset \bar{S}} c_{S,T} \chi_T(M).
\]
\end{Theorem}

\begin{remark}
Since $c_{S,S}$ is defined to be $1$, if we pick any total order refining the partial order on strata give by the closure relation we see that the matrix $(c_{S,T})$ is unitriangular, and so the sets of numbers $\{\chi_S(M)\}_{S \in \mathcal S}$ and $\{m_S(M)\}_{S \in \mathcal S}$ determine each other. Thus one can invert the above theorem to give a formula for the multiplicities of the characteristic cycle in terms of the local Euler characteristics. This is the form of the theorem stated in \cite{K73}. In fact the map from holonomic $\mathcal D$-modules to constructible functions given by the $\chi_S$s and the characteristic cycle map from holonomic $\mathcal D$-modules to Lagrangian cycles induce isomorphisms from the Grothendieck group to the group of constructible functions and Lagrangian cycles respectively. The index theorem is then an explicit geometric description of the induced isomorphism between the two abelian groups. Note that the difference in signs between our statement of the index theorem and that in \cite{EM} arises as we work with $\mathcal D$-modules rather than constructible sheaves.
\end{remark}

\end{document}